\documentclass[smallextended]{svjour3}       
\smartqed  
\usepackage{graphicx}
\usepackage[english]{babel}
\usepackage{graphics}
\usepackage{graphicx}
\usepackage{amssymb}
\usepackage{arydshln}
\usepackage{amsmath}
\usepackage{amsfonts}
\usepackage{float}
\usepackage{multirow}
\usepackage{booktabs}
\usepackage{pstricks,pst-plot}
\usepackage{hyperref}
\usepackage{epstopdf}
\usepackage{subfigure}
\usepackage{dsfont}
\usepackage{caption}
\usepackage{enumerate}

\newcommand{\be}{\begin{equation}}
\newcommand{\ee}{\end{equation}}
\newcommand{\ba}{\begin{array}}
\newcommand{\ea}{\end{array}}

\newcommand{\comment}[1]{}

 \newcommand{\B}[1]{\mbox{\boldmath $#1$}}

\newcommand{\rank}{\mbox{\rm {rank }}}

\newcommand{\Rc}{\raisebox{-0.11cm}[0.1cm][0.0cm]{\mbox{$\Rsh$}}}
\newcommand{\rshno}{\rotatebox[origin=c]{180}{\reflectbox{\Rc}}}
\newcommand{\rc}{\raisebox{0.11cm}[0.1cm][0.0cm]{\mbox{\rshno}}}

\newcommand{\ca}{\hspace{-8.5pt}
\makebox[0.3cm]{\raisebox{-0.2cm}[0.1cm][0.1cm]{{${\hookrightarrow}$}}}}

\newcommand{\stb}{\curvearrowright}
\newcommand{\stbt}[1]{\mbox{$\curvearrowright$\hspace{-8pt}\raisebox{0.22cm}[0cm][0cm]{$\scriptstyle #1$}\hspace{4pt}}}

\newcommand{\sltt}[1]{\mbox{\reflectbox{\rotatebox[origin=t]{180}{$\curvearrowleft$}}\hspace{-8pt}\raisebox{-0.1cm}[0cm][0cm]{$\scriptstyle #1$}\hspace{4pt}}}
\newcommand{\slb}{\reflectbox{$\curvearrowright$}}
\newcommand{\slbt}[1]{\mbox{\reflectbox{$\curvearrowright$}\hspace{-6pt}\raisebox{0.22cm}[0cm][0cm]{$\scriptstyle #1$}\hspace{4pt}}}
\newsavebox{\combarrow}
\savebox{\combarrow}[0.4cm]
{
\Rc\hspace{-8.5pt}
 \raisebox{-0.4cm}{\rc}
}

\title{Fast QR iterations for unitary plus low rank matrices}

\author{R. Bevilacqua
\and G. M. Del Corso\and L.~Gemignani\thanks{The research of the last two authors was partially supported by GNCS project ``Analisi di matrici sparse e data-sparse: metodi numerici ed applicazioni''and by the project sponsored by  University of Pisa
	under the grant PRA-2017-05.} }

\institute{R. Bevilacqua, G. M. Del Corso L. Gemignani\at
	Dipartimento di Informatica, Universit\`a di Pisa,
	Largo Bruno Pontecorvo 3, 56127 Pisa, Italy,
	\email{\{roberto.bevilacqua,gianna.delcorso,luca.gemignani\}@unipi.it}           
}

\begin{document}

\maketitle

%

\begin{abstract}
	Some  fast  algorithms for  computing  the eigenvalues
	of a  (block) companion matrix  have  recently appeared in the literature.
	In this paper we  generalize the approach to encompass unitary plus low rank
        matrices of the form $A=U + XY^H$ where $U$ is
	a general unitary matrix. Three important cases for applications are $U$ unitary diagonal,
        $U$ unitary block Hessenberg and $U$ unitary in  block  CMV form.   Our extension exploits
        the properties of a larger  matrix $\hat A$ obtained by a certain embedding of the Hessenberg reduction of $A$
        suitable to maintain its structural properties.  We show that $\hat A$ can be factored as 
	product of  lower and upper  unitary Hessenberg matrices  possibly perturbed in the first $k$ rows, and, moreover, 
         such a data-sparse  representation is well suited  for  the design of fast eigensolvers based on the
	QR  iteration. The resulting algorithm is  fast  and backward stable.

\end{abstract}


\section{Introduction}
The design  of specialized algorithms that compute the eigen\-va\-lues of unitary matrices is  so far a classical topic
in structured numerical linear algebra
(compare  \cite{BE} and the references given therein).  The major applications that stimulate research in this
area lie  in signal processing \cite{AGR},  in time series analysis \cite{ACR},  in Gaussian quadrature on the unit circle
\cite{GWB} and
in trigonometric approxi\-ma\-tion theory \cite{FH}.  In the last years many authors have dealt with the  issue  of efficiently
computing the eigenvalues of unitary matrices perturbed by low rank corrections (see the books \cite{Yuli_book,Raf_book}
for general overviews  of these   developments).  Motivations come from (matrix) polynomial  root-finding  problems 
\cite{ACR_0,Bindel2005AFA,BDG,BGP_NM,Ge05} and generally
from   eigenvalue problems associated with finite truncations of large (block) unitary matrices arising in the aforementioned
applications  \cite{SVA} as well as in certain statistical methods for the analysis of complex systems \cite{FS03}.
Typically in these applications large unitary matrices  are represented in condensed form using the  (block) Hessenberg
\cite{Gragg}
or the (block)  CMV shape \cite{AGR86,Ki85}.

The papers  \cite{BEGG_simax,CG}  presented the first fast and numerically reliable 
eigensolvers for certain low rank perturbations of unitary matrices, while in~\cite{VDC10} the analogous case of low rank perturbation of Hermitian structure is addressed.
Since then two challenging issues have attracted much  work: 1) the search of numerical algorithms  that are  computationally
efficient with respect to the size  both of the matrix  and  of  the perturbation  and 2)  a formal proof of the backward stability of these algorithms. 
Very recently numerical methods which combine all these two  features  have been proposed  in \cite{AMVW15,AMRVW17} for
computing the eigenvalues of companion and block companion matrices, respectively.  These methods incorporate
some techniques that are specifically adjusted to exploit the properties of companion and block companion forms.   In particular,
the Hessenberg reduction of a block companion matrix is found by  relying upon the decomposition of the matrix
as product of scalar companion matrices  which provides  the  factored representation of the Hessenberg reduction
to be used in the QR iterative process.


In this paper we generalize the approach  pursued in \cite{AMVW15,AMRVW17}
to deal with
 input matrices   of the form  $A_i=U + XY^H\in \mathbb C^{n\times n}$ where
$U$ is  a general $n\times n$ unitary  matrix and $X, Y \in \mathbb{C}^{n \times k}$ with $k\leq n$. Eigenvalue computation is
customarily a two-step process.  Firstly the  input matrix $A_i$ is reduced in Hessenberg form by  unitary similarity, that is,
$A_i\rightarrow A_f\colon=QA_iQ^H$, where    the  final $A_f$ is Hessenberg and $Q$ unitary, and
then the QR iteration is applied to the
Hessenberg reduction $A_f$  for  computing its Schur form.  Each  iterate generated  by a  fast adaptation of the QR
scheme   inherits xthe  condensed representation of the   initial matrix that is the
matrix $A_f$  obtained  at the end of the Hessenberg reduction.  By setting $U\colon=QUQ^H$, $X\colon=QX$ and $Y=QY$
it is found that this matrix $A_f=U + XY^H$ is still unitary plus low rank in Hessenberg form.

The efficient computation
of a condensed representation of such matrix is the subject of the  papers \cite{GR,BDCG_TR}  where it is shown that
$A_i$ can be   embedded into  a larger matrix  $\hat A_i$ which is
converted by unitary similarity in the Hessenberg  matrix  $\hat A_f$ specified in factored form as the product of
three  factors, that is $\hat A_f=L\cdot F\cdot R$ with  $L$ and $R$ unitary  $k$-Hessenberg matrices  and
$F$ unitary upper Hessenberg matrix perturbed
in the first $k$
rows. The construction greatly simplifies  when the  matrix $A_i$ is a unitary (block) Hessenberg or CMV matrix
modified in the first/last rows/columns since the three factors $L, F$ and $R$ can be directly obtained simplifying the initial transformation $A_i\to A_f$ as explained in~\cite{BDCG_TR}. In particular, this is the case of block companion matrices, for which we can compute the factored Hessenberg representation with $O(n^2k)$ operations. 

Our present work   aims  at designing a fast  version of the implicit QR eigenvalue method~\cite{Francis} for unitary plus low rank  Hessenberg
matrices
 $\hat A=\hat A_f=U + XY^H$, $U\in \mathbb{C}^{n \times n}$ unitary   and $X, Y \in \mathbb{C}^{n \times k}$,
represented  in compressed form  as $\hat A=L\cdot F\cdot R$, 
where $L$ is the product of $k$ unitary lower Hessenberg matrices, $R$
is the product of
$k$  unitary upper Hessenberg matrices and the middle factor $F$  is a unitary upper Hessenberg matrix perturbed
in the first $k$
rows. The representation  is data-sparse since it  involves $O(nk)$  data storage  consisting of
$O(k)$ vectors of length
$n$ and $O(nk)$  Givens rotations. Specifically, the main  results  are:
\begin{enumerate}
\item The development of a bulge-chasing technique for performing one step of the implicit QR algorithm applied to a matrix
  specified in the $LFR$ format by returning as output the updated factored  $LFR$ representation of the new iterate.
\item A careful  look   at  the structural properties of Hessenberg matrices given  in the $LFR$ format by implying that,
  under some auxiliary assumptions on the properness of the factors $L$ and $R$, the  middle  matrix $F$ is reducible iff the same
  holds for the  Hessenberg matrix.  It follows that the deflation in the Hessenberg iterate  can be revealed in the
middle factor  converging to an upper triangular matrix in the limit. 
\item A cost and error analysis of the resulting adaptation of the implicit QR algorithm.  We prove that 
one single QR iteration requires   $O(n k)$ ops only and it is  backward stable.
\end{enumerate}

The paper is organized as follows.   In Section~\ref{two} we recall some preliminary material about
the structural properties of possibly perturbed unitary matrices. Section~\ref{sec:rap} gives  the theoretical
foundations of our algorithm which is presented and analyzed  in  Section~\ref{sec:alg}.
In Section~\ref{sec:back} the backward stability of the algorithm is formally proved. Finally,  in Section~\ref{sec:numexp}
we show the results of numerical experiments followed by some conclusions and future work in Section~\ref{sec:conclusions}.
 
\section{Preliminaries}\label{two}

We  first recall some basic properties of unitary matrices which
play an important role in the derivation of our methods.
\begin{lemma}\label{l1}
Let $U$ be a unitary matrix of size $n$. Then 	
$$
\rank(U(\alpha, \beta))=\rank(U(J\backslash \alpha,J\backslash \beta))+|\alpha|+|\beta|-n
$$
where $J=\{1, 2, \ldots, n\}$ and $\alpha$ and $\beta$ are subsets of $J$.
If $\alpha=\{1, \ldots, h\}$ and $\beta=J\backslash \alpha$, then we have
$$
\rank(U(1:h, h+1:n))=\rank(U(h+1:n, 1:h)),\quad  \mbox{ for all } h=1, \ldots, n-1.
$$

\end{lemma}
\begin{proof}
  This well known symmetry in the  rank-structure of unitary matrices
  follows by a straightforward application of the nullity theorem~\cite{FM}.\qed
\end{proof}

\begin{lemma} \label{gant}
Let $U$ be a unitary matrix of size $n$, and let $\alpha, \beta\subseteq\{1, 2, \ldots, n\}$, such that $|\alpha|=|\beta|$, then
$$
|\det(U(\alpha, \beta))|=|\det(U(J\backslash \alpha,J\backslash \beta))|.
$$	
	\end{lemma}
\begin{proof}
See Gantmacher~\cite{Ga60}, Property 2 on page 21.	\qed
\end{proof}

\begin{definition}
	A matrix $H$ is called {\em $k$-upper Hessenberg} if $h_{ij}=0$ when $i>j+k$.
	Similarly, $H$ is called  {\em $k$-lower Hessenberg} if $h_{ij}=0$ when $j>i+k$.   In addition, 	
	when  $H$ is {\em $k$-upper Hessenberg} ({\em $k$-lower Hessenberg}) and
        the outermost entries are non-zero, that is, $h_{j+k,j}\neq 0$ ($h_{j,j+k}\neq 0$), $1\leq j\leq n-k$,
         then  the matrix is called {\em proper}.	
\end{definition}

Note that for $k=1$, that is when the matrix is in
Hessenberg form, the notion of properness  coincides with that of  being unreduced.
Also,  a $k$-upper Hessenberg matrix $H\in \mathbb C^{n\times n}$ is proper iff $\det(H(k+1:n, 1:n-k))\neq 0$. Similarly a $k$-lower Hessenberg matrix $H$ is proper iff $\det(H(1:n-k, k+1:n))\neq 0$. For $k<0$ a $k$-Hessenberg matrix is actually a strictly triangular matrix with $-k$ vanishing diagonals.

It is well known~\cite{Watkins} that, given a  non-zero  $n$-vector $\B x$ we can build a zero creating matrix
from a product of $n-1$ Givens matrices ${\cal G}_1\cdots {\cal G}_{n-1}$, where ${\cal G}_i=I_{i-1}\oplus G_i \oplus I_{n-i-1}$
and $G_i$ is a $2\times 2$ complex Givens rotations  of the form $\left[ \begin{array}{cc} c &-s\\s& \bar c \end{array}\right]$
such that $|c|^2+s^2=1$, with $s\in \mathbb{R}, s\ge 0$. The subscript index $i$ indicates the active part of the matrix
${\cal G}_i$.
The descending sequence of Givens rotations $H={\cal G}_1\cdots {\cal G}_{n-1}$ turns out to be a unitary upper Hessenberg
matrix such that $H{\B x}=\alpha {\B e}_1$, and $|\alpha|=\| \B x\|_2$.
Note that  $H$ is proper if and only if all the Givens matrices appearing in its factorization are non trivial, i. e. $s\neq 0$~\cite{AMVW15}. Generalizing this result we obtain the following lemma.
\begin{lemma}  \label{prop-v}
Let $X\in \mathbb{C}^{m\times k}$, $k<m$,  be of full rank. Then
\begin{enumerate}
	\item there exist a unitary $k$-upper Hessenberg matrix $H$  and an upper triangular matrix $T\in \mathbb{C}^{m\times k},  T=\left[\begin{array}{c} T_k\\0\end{array}\right]$ with $T_k\in \mathbb{C}^{k\times k}$ nonsingular such that
\begin{equation} \label{vx}
HX=T.
\end{equation}
	\item 	The product of the outermost entries of $H$ is given by
	\begin{equation} \label{prod_h}
	\prod_{i=1}^{m-k} |h_{i+k, i}|=\frac{|\det(X(m-k+1:m, 1:k))|}{\prod_{i=1}^k \sigma_i(X)},
	\end{equation}
	where $\sigma_1(X), \sigma_2(X), \ldots, \sigma_k(X)$ are the singular values of $X$.
      \item  Let $s$ be the maximum index such that $\rank(X(s:m, :))=k$, then $h_{i+k, i}\neq 0$ for $i=1, \ldots, s-1$.
\end{enumerate}
\end{lemma}	
\begin{proof}
  The existence of $H$ is proved by construction.
  Relation \eqref{vx} defines a QR decomposition of the matrix $X$. The unitary factor $H$ can be  determined as product
  of $k$ unitary  upper Hessenberg matrices $H= H_{k}\cdots H_1$ such that  $H_h=I_{h-1}\oplus \tilde H_h$, $H_0=I_m$ and
  $\tilde H_h X^{(h-1)}(h:m,h)=t_{hh}{\bf e}_1$ where  $X^{(h-1)}=H_{h-1}\cdots H_1 H_0\cdot X$, $1\leq  h\leq k$.

Now,  let us split $H$ into  blocks
$$
H=\left[
\begin{Array}{cc}
H_{11}&H_{12}\\H_{21}& H_{22}\end{Array}
\right],
$$
where $H_{12}$ is $k\times k$ and $H_{21}$ is $(m-k)\times (m-k)$, upper triangular.
The product of the outermost entries of $H$ is given by $\det(H_{21})=\prod_{i=1}^{m-k} h_{i+k,i}$.
Since $H$ is unitary, and $X=H^HT$ we have
$$
\det(X(m-k+1:m, :))=\det(H_{12}^H) \det (T_k).
$$
From Lemma~\ref{gant} we have
$$
|\det(H_{21})|=|\det(H_{12})|=\frac{|\det(X((m-k+1:m, :))|}{|\det(T_k)|}.
$$
We get  relation \eqref{prod_h}
observing that if $X=P\Sigma Q^H$ is the SVD decomposition of $X$, $(HP)\Sigma Q^H$ is the SVD decomposition of
$T$  and hence $\sigma_i(T_k)=\sigma_i(X)$,  $i= 1, \ldots, k$.


Finally,  let  $s$ be  the maximum index such that $\rank(X(s:m, 1:k))=k$.   Then $s\le m-k+1$ and,  moreover,  from
$X(s:m, 1:k)=H^H(s:m, 1:k)T_k$ we obtain that  $\rank(H(1:k, s:m))=k$ since $T_k$ is nonsingular. Using Lemma \ref{l1}
we have $k=\rank(H^H(s:m, 1:k))=\rank(H( k+1:m, 1:s-1)+k+(m-s+1)-m$, meaning that $H( k+1:m, 1:s-1)$ has full rank equal to $s-1$.
Since $H( k+1:m, 1:s-1)$ is upper triangular we have  that   $h_{i+k, i}\neq 0$ for $i=1, \ldots, s-1$. \qed
\end{proof}	

\begin{remark} \label{rem-h}
  From the proof of Lemma~\ref{prop-v} we know that $H$ can be written as a product of $k$ upper Hessenberg matrices,
  i.e., $H=H_k H_{k-1}\cdots H_1$. The $j-th$  of these Hessenberg matrices is the one annihilating the $j$-th column of $X^{(j-1)}$
  from row $j+1$ to row $m$. Then each $H_j$ can be factored as the product of $m-j$ Givens rotations. From this observation
  we get that $H_j={\mathcal G}_{j}^{(j)}\cdots {\mathcal G}_{m-1}^{(j)}$ where each ${\mathcal G}_i^{(j)}$ is a Givens rotation
  acting on rows $i, i+1$.  This decomposition of $H$ corresponds  to annihilate progressively the  lower subdiagonals of $H$ by
means of rotations working on the left. Alternatively, we can proceed by zeroing the  lower subdiagonals of $H$ by
means of rotations working on the  right and acting on the columns of $H$. In this way we  find a different
factorization of the form $H=D\hat H_k \hat H_{k-1}\cdots \hat H_1$  where  $\hat H_j= \hat{\cal G}_1^{(j)}\hat{\cal G}_{2}^{(j)}\cdots \hat{\cal G}_{m-k+j-1}^{(j)}$ and $D$ is unitary diagonal.
\end{remark}

\section{Representation} \label{sec:rap}

Generalizing the approach discussed in~\cite{AMVW15} for the companion matrix, it is useful to embed the  unitary plus low-rank
matrix $A$ into a larger matrix  to guarantee the properness of some factors of the representation that we are going to introduce.



\begin{theorem}\label{theo:embedding}
Let $A\in \mathbb{C}^{n\times n}$ be such that
$A=U+X\, Y^H, $ with $U$ unitary and $X, Y$ ${n\times k}$ full rank matrices.
We can construct an $N\times N $ matrix $\hat A$, $N=n+k$,  such that
$\hat A=\hat U+ \hat X\hat Y^H$,
with $\hat U$ unitary, $ \hat X, \hat Y$  $N\times k$ full rank matrices,  $\hat X(n+1:N,  :)=-I_k$, and such that
\begin{equation} \label{hatr}
\hat A=\left[\begin{array}{cc}
A &B\\
0_{k, n}&0_k\end{array}\right], \quad \mbox{for a suitable } B\in \mathbb{C}^{n\times k}.
\end{equation}
\end{theorem}
\begin{proof}
The proof is constructive. We first compute the economy size QR decomposition of matrix $Y$, $Y=QR$ where $Q\in \mathbb{C}^{n\times k}$ and $R\in \mathbb{R}^{k\times k}$. Set $\tilde Y=Q$ and $\tilde X= XR^H$. We still have $XY^H=\tilde X \tilde Y^H$ but now $Y$ has orthonormal columns, i.e., $\tilde Y ^H \tilde Y=I_k$.
Define
\begin{equation}\label{uhat}
\hat U=\left[\begin{array}{cc}
U-U\tilde Y \tilde Y^H & B\\
\tilde Y^H &0_k\end{array}\right],
\end{equation}
where $B=U\tilde Y$
and
\begin{equation}\label{xhat}
\hat X=\left[\begin{array}{c}
\tilde X+B\\
-I_k \end{array}\right]\quad
\hat Y=\left[\begin{array}{c}
\tilde Y\\
0_k \end{array}\right].
\end{equation}
Note that $\hat U+\hat X\hat Y^H$ has the structure described in~\eqref{hatr} and,  moreover by direct calculation
we can verify that  $\hat U$ is unitary.\qed
\end{proof}
%
From now on we denote by $N=n+k$ the dimension of  the matrix $\hat A$.
It is worth pointing  out that  in view of  the block triangular structure in
\eqref{hatr}  the Hessenberg reduction of the original matrix $A$ can be easily
adjusted to specify the  Hessenberg reduction of  the larger matrix $\hat A$. Thus,  in the following
it is always assumed that both
$A$ and $\hat A$ are in upper Hessenberg form.
%
%
\begin{theorem}\label{teo:rep}
Let $\hat{A}=\hat U+\hat X\hat Y^H\in \mathbb{C}^{N\times N}$ be the upper Hessenberg matrix obtained by embedding an $n\times n$ proper  Hessenberg matrix $A$ as  described in Theorem~\ref{theo:embedding}.
Then we can factorize $\hat A$ as follows
\begin{equation} \label{form:rapp}
\hat{A}=L\cdot F\cdot R, \quad {\mbox{where}}
\end{equation}
\begin{description}
\item $L$ is a proper unitary $k$-lower Hessenberg matrix.
\item $R$ is a unitary $k$-upper Hessenberg  matrix. Moreover, the leading $n-1$ entries in the outermost diagonal of $R$, $r_{i+k,i},i=1,\ldots,n-1$, are nonzero.
\item $F=Q+TZ^H$, where $Q$ is a block diagonal unitary Hessenberg matrix, $Q=\left[\begin{Array}{c|c}
I_k & \\ \hline & \hat Q
  \end{Array}\right] $, with $\hat Q$ proper, $T=\left[\begin{array}{c}T_k\\ 0_{n,k}\end{array} \right]$ with $T_k\in \mathbb{C}^{k\times k}$ upper triangular, and $Z\in \mathbb{C}^{N\times k}$, with full rank.
\end{description} 
\nopagebreak If in addition $A$ is nonsingular then $R$ is proper.
\end{theorem}
\begin{proof}
First note that from the properness of $A$ it  follows that  $\rank(A)\ge n-1$.  	
From Theorem~\ref{theo:embedding} we have that $\hat X$ has full rank, and $\det(\hat X(n+1:N, :))=\det (-I_k)\neq 0$, hence,
by Lemma~\ref{prop-v} we can find a proper $L^H$ and a nonsingular square triangular $T_k$ such that $L^H\hat X=T$, with $T=[T_k^T, 0_{k,n}]^T$.
For the properness of $L^H$ and $A$,  we get that $L^H\hat A$ is a proper $(k+1)$-upper Hessenberg matrix and
moreover the  matrix $V=L^H\hat U=L^H\hat A-T\hat Y^H$, is unitary and  still a proper $(k+1)$-upper Hessenberg matrix because $T\hat Y^H$ is null under the $k$-th row.

Now the matrix $V$ can be factored as $V=QR$, where $R$ is unitary $k$-upper Hessenberg, and $Q^H$ is the unitary lower Hessenberg matrix obtained as the product of the $n-1$ Givens rotations annihilating from the top  entries in the outermost diagonal of $V$, i.e., $Q^H={\cal G}_{N-1}\cdots {{\cal G}_{k+2}\cal G}_{k+1}$, where ${\cal G}_i$ acts on rows $i,i+1$. Since the first $k$ rows are not involved, the matrix $Q$ has the structure $Q=\left[\begin{Array}{c|c}
I_k & \\ \hline & \hat Q
  \end{Array}\right] $,
where $\hat Q$ is unitary $n\times n$ Hessenberg. Moreover, since $V$ is proper, $\hat Q$ is proper as well.

From the definitions of $V$, $Q$ and $R$ we have:
$$
\hat{A}=LV+LT\hat Y^H=L(Q+TZ^H)R,
$$
where $Z=RY$. The matrix $Z$ is full rank, since $R$ is unitary and $Y$ is full rank.

Now let us consider the submatrices $R(k+1:N,1:j)$, for $j=n-1$ and $j=n$. In both cases, from the relation $R=Q^H(L^H\hat A-TY^H)$ and the structural properties of the matrices involved therein, we have that 
\begin{eqnarray*}\rank (R(k+1:N,1:j))
&=&\rank(\hat Q^H L^H(k+1:N, 1:n)\hat A(1:n, 1:j) )\\
&=&\rank (A(1:n,1:j)).
\end{eqnarray*}
For $j=n-1$, since $A$ is proper, the rank of that submatrix  is $n-1$. This implies that the entries $r_{i+k,i},i=1,\ldots,n-1$, are nonzero.
For $j=n$, if $A$ is nonsingular, then the rank is $n$, so $r_{N,n}$ is nonzero. \qed

\end{proof}

The following theoremproves that the product of the factors $L, F, R$  having the properties
stated in Theorem~\ref{teo:rep} is indeed an upper Hessenberg matrix with the last $k$ rows equal to zero. It reveals also that
deflation can be performed  only when one of the subdiagonal entries  of $Q$ approaches zero.

\begin{theorem}\label{teo:rep1}
Let $L, R\in \mathbb{C}^{N\times N}$, where $L$ is a proper unitary $k$-lower Hessenberg matrix and $R$ is a unitary $k$-upper Hessenberg matrix. Let $Q$ be a block diagonal unitary upper Hessenberg matrix of the form $Q=\left[\begin{Array}{c|c}
I_k & \\ \hline & \hat Q
  \end{Array}\right] $, with $\hat Q$ $n\times n$ unitary Hessenberg and $T_k$ a $k\times k$ nonsingular upper triangular matrix.
Then
\begin{enumerate}
\item $L(n+1:N, 1:k)=-T_k^{-1}$.
\item Setting $Z^H=L(n+1:N, :)Q$,  $T=[T_k^H, 0]^T$, and $F=Q+TZ^H$, we have that
\begin{enumerate}
	\item  the matrix $\hat A=LFR$
	  is  upper Hessenberg, with $\hat A(n+1:N, :)=0$, that is
          $$\hat A=\left[\begin{array}{cc} A & \ast \\ 0_{k, n} & 0_{k,k}\end{array}\right],$$
	\item $\hat A$ is a unitary plus rank $k$ matrix.
	
\end{enumerate}
\item If $R$ is proper then the upper Hessenberg matrix $A\in \mathbb{C}^{n\times n}$ is nonsingular. In this case $A$ is proper if and only if $\hat Q$ is proper.
\end{enumerate}
\end{theorem}
\begin{proof}
To prove part 1, note that $\hat X=LT$, and $\hat X(n+1:N, :)=-I_k$, hence $-I_k=L(n+1:N, :)T=L(n+1:N, 1:k) T_k$, and then we have $L(n+1:N, 1:k)=-T_k^{-1}$.

For  part 2, let us consider the   matrix $C=L\,Q$. This matrix
is unitary with a $k$-quasiseparable structure below the $k$-th upper diagonal. In fact, for any $h, h=2, \ldots n+1$ we have
\begin{eqnarray*}
C(h:N, 1:h+k-2)&=&L(h:N, :)\, Q(:, 1:h+k-2)=\\
&=&L(h:N, 1:h+k-1)\,Q(1:h+k-1,1:h+k-2).
\end{eqnarray*}
Applying Lemma~\ref{l1} we have $\rank(L(h:N, 1:h+k-1))=k$, implying that $\rank(C(h:N, 1:h+k-2))\le k$. 

Since $C(n+1:N, 1:k)=L(n+1:N, 1:k)$ is nonsingular, we conclude that $\rank(C(h:N, 1:h+k-2))= k$.
From this observation we can then find  a set of generators $P, S\in \mathbb{C}^{(N\times k)}$ and
a $(1-k)$-upper Hessenberg matrix $U_k$
 such that $U_k(1,k)=U_k(n, N)=0$
so that $C=PS^H+U_k$
(see~\cite{BDC04,Yuli_book}).  Moreover, we
have $C(n+1:N, 1:k)=L(n+1:N. : ) Q(:, 1:k)=M$. Then we can recover the rank $k$
correction $PS^H$ from the left-lower  corner of $C$  obtaining
$$
PS^H=-C(:, 1:k) T_k C(n+1:N, :).
$$
Since $C(:, 1:k)=L Q(:, 1:k)=L(:,1:k)$, we get that $PS^H=-LTZ^H$, and hence $U_k=L(Q+TZ^H)=L\, F$.
We conclude the proof of part (b), by noticing that $\hat A=U_k\, R$ is upper Hessenberg as it is  the product of a $(1-k)$-upper
Hessenberg matrix by  a $k$-upper Hessenberg matrix. Moreover, we find that
$\hat A(n+1:N, :)=U_k(n+1:N, :)R=0$ since $U_k(n+1:N, :)=0$.

To prove part 3, as already observed in the proof of Theorem~\ref{teo:rep}, we use the rank equation
$$\rank (R(k+1:N,1:n))=\rank  (\hat A(k+1:n,1:n))=\rank (A),$$
thus, if $R$ is proper then $A$ is nonsingular. In this case, from the properness of $L$ and noticing that
\begin{equation}
\label{adefl}
a_{i+1,i}=q_{i+k+1,i+k}\,r_{i+k, i}/\bar{l}_{i+1,i+k+1}, \quad i=1, \ldots, n-1,
\end{equation}
we get that $a_{i+1, i}=0$ iff $q_{i+k+1, i+k}=0$.\qed
\end{proof}

\begin{remark} From the previous Theorem, one sees that when a matrix $\hat A$ is represented in the $LFR$ form where $L,F$ and $R$ have the structural properties required, then $A$ is nonsingular if and and only if $R$ is proper.
  Moreover, from~\eqref{adefl} one deduces that one of the outermost entries $a_{i+1, i}$ can be zero only if we  have  either
  $q_{i+k+1,i+k}=0$ or $r_{i, i+k}=0$. Vice-versa, we can have that $r_{N,n} =0$ without
  any subdiagonal entry of $A$ being equal to zero. This is the only case where $A$ is proper and singular.
\end{remark}

The next theorem   shows that the  compressed representation $\hat A=LFR$ is eligible to be used  under the QR eigenvalue algorithm for computing the eigenvalues of $\hat A$ and, a fortiori, of $A$.

\begin{theorem}\label{teo-qr}
Let $\hat A=\hat U+\hat X\hat Y^H\in \mathbb{C}^{N\times N}, N=n+k$ be a Hessenberg matrix obtained by embedding a proper $n\times n$ Hessenberg matrix $A=U+XY^H$ as described in  Theorem~\ref{theo:embedding}.
Let $P$ be the unitary factor of the QR factorization of $p_d(\hat A)$, where $p_d(x)$
is a monic polynomial of degree $d$.  Let
$\hat A^{(1)}=P^H\hat A P$
be the matrix obtained by applying a step of the multi-shifted QR algorithm to  the matrix $\hat A$ with shifts
being the roots of $p_d(x)$.  Then, we have that
$$
\hat A^{(1)}=\left[\begin{array}{cc}
A^{(1)} &B^{(1)}\\
0_{k, n}&0_k\end{array}\right],
$$
where $A^{(1)}$ is the matrix generated by applying one
step of the multi-shifted QR algorithm to  the matrix $A$ with shifts
being the roots of $p_d(x)$.  Both $\hat A^{(1)}$  and $A^{(1)}$  are upper Hessenberg and if  $A^{(1)}$   is proper then
the factorization of $\hat A^{(1)}=L^{(1)}\, F^{(1)}\, R^{(1)}$
exists and has still the same properties stated in Theorem~\ref{teo:rep};  in particular,  $L^{(1)}$ is proper and, if $A$ is nonsingular also $R^{(1)}$ is proper.
\end{theorem}
\begin{proof}
From~\eqref{hatr} we have
$$
\hat A=\left[\begin{array}{cc}
A &B\\
0_{k, n}&0_k\end{array}\right].
$$
 Since $p_d(\hat A)$ is also block triangular, we  can take
\begin{equation}
\label{P}
P= \left[\begin{array}{cc}
P_1 &0_{n,k}\\
0_{k, n}&P_2\end{array}\right],
\end{equation}
where  $P_1$ and $P_2$ are unitary. Hence,
$$
\hat A^{(1)}=\left[\begin{array}{cc}
A^{(1)} &B^{(1)}\\
0_{k, n}&0_k\end{array}\right],
$$
where $A^{(1)}$ is the matrix generated by applying one
step of the multi-shifted QR algorithm to  the matrix $ A$ with shifts
being the roots of $p_d(x)$.
We have $\hat A^{(1)}=P^H\hat A P=P^H\hat UP+P^H\hat X\, \hat Y^HP=U_1+\hat X_1\hat {Y}_1^{H}$, setting $U_1=P^H\hat UP$ and $\hat X_1=P^H\hat X, \hat{Y}_1=P^H \hat Y$. Because $P_2$ is unitary, we have that $|\det(\hat X_1(n+1:N, :))|= |\det(\hat X(n+1:N, :))|\neq 0$, then the conditions given by Lemma~\ref{prop-v} are satisfied and we can conclude that $L^{(1)}$ is proper.
We note that $\hat A^{(1)}$ and $ A^{(1)}$  are upper Hessenberg for the well known properties of the shifted QR algorithm.
When $ A^{(1)}$  is proper then   we can apply Theorem~\ref{teo:rep} which guarantees the existence of the representation of $\hat A^{(1)}$.\qed
\end{proof}	

The algorithm we propose is an implicitly shifted QR method, and hence the factors $L^{(1)}, F^{(1)}, R^{(1)}$ are obtained by manipulating Givens rotations. In Section~\ref{sec:alg} we describe the algorithm and we show that the factors obtained with the implicit procedure agree with the requirements given in Theorem~\ref{teo:rep1}. The implicit Q-Theorem~\cite{MC} guarantees that the matrix obtained after an implicit step is  basically the same matrix one get with an explicit one.  The next result gives a quantitative measure of the properness of  matrices $L$ and $R$ generated along the QR iterative method. 

\begin{corollary} \label{cor_prod}
Let $\hat U, \hat X, \hat Y$ as described in Theorem~\ref{theo:embedding} and let $\hat A= L\,F\,R$ as in Theorem~\ref{teo:rep}. Let $K=1/\prod_{i=1}^k \sigma_i(X)$, where $\sigma_i(X)$ are the singular values of $X$.
We have:
\begin{enumerate}
	\item  the module of the product of the outermost entries of $L$, is such that $\prod_{i=1}^n |l_{i, i+k}|=K$ and is constant over QR steps. Moreover for each outermost entry of $L$  we have $K\le |l_{i, i+k}|\le 1$.
		\item the module of the product of the outermost entries of $R$ is $\prod_{i=1}^n |r_{i+k,i}|=K |\det A|$  and is constant over QR steps. Moreover for each outermost entry of $R$  we have $K{|\det(A)|}\le |r_{i+k, i}|\le 1$.
\end{enumerate}
\end{corollary}
%
%
%
\begin{proof}
To prove part 1 we first observe that  $|\det(\hat X(n+1:N, :))|=1$, because $\hat X(n+1:N, :))=-I_k$ by construction.
To prove that the product of the outermost entries remains unchanged over QR steps,  we use Theorem~\ref{teo-qr} observing that $|\det(\hat X(n+1:N, :))|=|\det(\hat X_1((n+1:N, :))|$ and that  $\hat X$ and $\hat X_1$ have the same singular values.
We get the thesis applying part 2 of Lemma~\ref{prop-v}.

We can also see that $0<|l_{j,j+k}|\le 1$ and that $|l_{j,j+k}|=K/|\prod_{i=1, i\neq j}^{n} l_{i,i+k}|$.
Since $|\prod_{i=1, i\neq j}^{n} l_{i,i+k}|\le 1$ we have $|l_{j,j+k}|\ge K$.

The relation on $\prod_{i=1}^n |r_{i+k, i}|$ is similarly deduced by applying Binet rule to equality $L(k+1:N, 1:n)A=\hat Q R(k+1:N, 1:n)$.  After a QR step
the first $k$ rows of $V_1=L^{(1)H} \hat U^{(1)}$ are orthonormal  and, moreover,
the $k\times k$ submatrix in the right upper corner of $V_1$ satisfies
\begin{flalign}
	&|\det(V_1(1:k, n+1:N))|=\det(V_1(k+1:N, 1:n))|= && \nonumber \\ \nonumber
	&\hspace{2cm}=|\det(L^{(1)}(1:n, k+1:N))|\, |\det (A^{(1)})|=K |\det(A)|.&&
\end{flalign}
\qed
\end{proof}

\begin{remark}\label{rem:zrec}
  As observed in~\cite{AMVW15,AMRVW17} also for our representation it is possible to
  recover the structure of the $N\times k$ matrix $Z$ from the representation~\eqref{form:rapp}.
  In fact, we have $\hat A(n+1:N, :)=L(n+1:N, :)(Q+TZ^H)R=0_{k, N}$.
  Since $R$ is nonsingular, and $L(n+1:N, 1:k)=-T_k^{-1}$ we have that 
\begin{equation} \label{form:Z}
 Z^H = L( n+1:N, :) Q.
\end{equation}
%
%

\end{remark}
%
%
\section{The Algorithm} \label{sec:alg}
In this section we show how to carry out a single step of  Francis's
 implicitly shifted QR algorithm acting directly on the representation of the matrix described in Section~\ref{sec:rap}.
 In the sequel  we assume  $R$ to be a proper $k$-upper Hessenberg matrix. In the view of the previous sections this means
 that $A$ is nonsingular. If, otherwise, $A$ is singular then  we can perform a step of the QR algorithm with zero shift
 to remove the singularity.  In this way the parametrization of $R$ is automatically adjusted to specify a
 proper matrix in its active part. 

It is convenient to describe the representation and the algorithm using a pictorial representation already  introduced  in several papers (compare
with~\cite{R_book} and the references given therein).  Specifically, the action of a Givens rotation acting on two consecutive rows of the matrix is depicted as 
$
 \begin{array}{c}
\Rc\\
\rc\end{array}
$.
A chain of descending two-pointed arrows as below
$$
\begin{array}{c@{\hspace{1mm}}c@{\hspace{1mm}}c@{\hspace{1mm}}c@{\hspace{1mm}}c@{\hspace{1mm}}c@{\hspace{1mm}}c@{\hspace{1mm}}c@{\hspace{1mm}}c@{\hspace{1mm}}c@{\hspace{1mm}}c@{\hspace{1mm}}c@{\hspace{1mm}}c@{\hspace{1mm}}c@{\hspace{1mm}}c@{\hspace{1mm}}c@{\hspace{1mm}}c@{\hspace{1mm}}c}
\Rc	&&   &      &     \\
\rc	&\Rc&   &      \\
	&\rc &\Rc    &\\
	&    & \rc &\Rc& && \\
&	&   & \rc & \Rc\\
&	&   & &\rc \\

\end{array}
$$
represents a unitary upper Hessenberg matrix (in this case a $6\times 6$ matrix).
Vice versa, since any  unitary Hessenberg matrix $H$  of size $n$ can be factored as $H={\cal G}_1(c_1){\cal G}_2(c_2)\cdots {\cal G}_{n-1}(c_{n-1})\, D$, where ${\cal G}_k(c_k)=I_{k-1}\oplus G_k(c_k) \oplus I_{n-k-1}$, $G_k(c_k)=\left[ \begin{array}{cc} c_k &s_k\\-s_k& \bar c_k \end{array}\right]$, with $|c_k|^2+s_k^2=1, s_k\ge 0$, and $D$ is unitary diagonal, we have that $H$ can be represented as follows
$$
\begin{array}{c@{\hspace{1mm}}c@{\hspace{1mm}}c@{\hspace{1mm}}c@{\hspace{1mm}}c@{\hspace{1mm}}c@{\hspace{1mm}}c@{\hspace{2.5mm}}}
\Rc	&\times&   &      &     \\
\rc	&\Rc&\times &      \\
&\rc &\Rc    &\times\\
&    & \rc &\Rc&\times&& \\
&	&   & \rc &\Rc&\times\\
&	&   &  &\rc& &\times\\
\end{array}
$$
where the $\times$ represent the entries of the diagonal phase matrix $D$. Similarly the chain 
$$
\begin{array}{c@{\hspace{1mm}}c@{\hspace{1mm}}c@{\hspace{1mm}}c@{\hspace{1mm}}c@{\hspace{1mm}}c@{\hspace{1mm}}c@{\hspace{1mm}}c@{\hspace{1mm}}c@{\hspace{1mm}}c@{\hspace{1mm}}c@{\hspace{1mm}}c@{\hspace{1mm}}c@{\hspace{1mm}}c@{\hspace{1mm}}c@{\hspace{1mm}}c@{\hspace{1mm}}c@{\hspace{1mm}}c}
&	&&   &       \Rc    \\
 &   & &\Rc   &\rc      \\
&	&\Rc &\rc    &\\
&\Rc&\rc    &  && && \\
\Rc&\rc&	&   &  & \\
\rc&&	&   &  & \\
\end{array}
$$
represents  a unitary  lower Hessenberg matrix.
As observed in Remark~\ref{rem-h}  the $k$-Hessenberg matrices $L$ and $R$ appearing in the representation of $\hat A$ can be factored as the product of  $k$ unitary Hessenberg matrices, and any unitary Hessenberg can be represented through their Schur parametrization~\cite{Gragg} 
by ascending or descending chains of Givens rotations times a unitary diagonal matrix. In our case the unitary  diagonal matrices that would be necessary to get the Schur parametrization in terms of Givens factors, can all be accumulated  in the unitary factor $Q$. In the light of Theorem~\ref{teo:rep} the careful reader will not be surprised by the shape of the chains  of Givens rotations in the factorization of factors $L$ and $R$ where some of the Givens rotations are missing. Hence, using our pictorial representations 
we can exemplify the case $n=6, k=3$, $N=n+k=9$,  as follows 
$$
\hat{A}=
\underbrace{\begin{array}{c@{\hspace{1mm}}c@{\hspace{1mm}}c@{\hspace{1mm}}c@{\hspace{1mm}}c@{\hspace{1mm}}c@{\hspace{1mm}}c@{\hspace{1mm}}c@{\hspace{1mm}}c@{\hspace{1mm}}c@{\hspace{1mm}}c@{\hspace{1mm}}c@{\hspace{1mm}}c@{\hspace{1mm}}c@{\hspace{1mm}}c@{\hspace{1mm}}c@{\hspace{1mm}}c@{\hspace{1mm}}c}
 &&   &      &     &\Rc&  & \\
  &&   &      &\Rc&\rc& \Rc & \\
 & &    &\Rc&\rc  &\Rc   &\rc  &\Rc \\
 & &\Rc&\rc&   \Rc  & \rc  & \Rc &\rc \\
  &\Rc&\rc&   \Rc  & \rc  & \Rc &\rc \\
\Rc&\rc&   \Rc  & \rc  & \Rc &\rc \\
  \rc&   \Rc & \rc  & \Rc &\rc \\
     & \rc  & \Rc &\rc \\
      &   & \rc & \\
\end{array}}_{L} \;\;
\left(
\underbrace{\begin{array}{c@{\hspace{1mm}}c@{\hspace{1mm}}c@{\hspace{1mm}}c@{\hspace{1mm}}c@{\hspace{1mm}}c@{\hspace{1mm}}c@{\hspace{1mm}}c@{\hspace{1mm}}c@{\hspace{1mm}}c@{\hspace{1mm}}c@{\hspace{1mm}}c@{\hspace{1mm}}c@{\hspace{1mm}}c@{\hspace{1mm}}c@{\hspace{1mm}}c@{\hspace{1mm}}c@{\hspace{1mm}}c}
& & & &\cdot	& \cdot  &\cdot & \times& \times & \times& \times&\times & \times&  \\
&&&	&\cdot  &\cdot  & \cdot& \times& \times & \times& \times&\times & \times&   \\
&&&	&\cdot & \cdot& \cdot  & \times& \times & \times& \times&\times & \times& \\
&&	& & & \phantom{\Rc}&\Rc &\times &   & & &   \\
&&	& & & & \rc&\Rc & \times & & &   \\
&&	& & & & & \rc&\Rc &\times & &   \\
&&	& & & & &   &\rc &\Rc &\times   \\
&&	& & & & &   & &\rc &\Rc& \times  \\
&&	& & & & & &    & &\rc & &\times \\
	\end{array}}_{Q+TZ^H}\;\; 
\right) \underbrace{
\begin{array}{c@{\hspace{1mm}}c@{\hspace{1mm}}c@{\hspace{1mm}}c@{\hspace{1mm}}c@{\hspace{1mm}}c@{\hspace{1mm}}c@{\hspace{1mm}}c@{\hspace{1mm}}c@{\hspace{1mm}}c@{\hspace{1mm}}c@{\hspace{1mm}}c@{\hspace{1mm}}c@{\hspace{1mm}}c@{\hspace{1mm}}c@{\hspace{1mm}}c@{\hspace{1mm}}c@{\hspace{1mm}}c}
&&\Rc&&&&&&&\\
&\Rc&\rc&\Rc&&&&&&\\
\Rc&\rc&\Rc&\rc&\Rc&&&&&\\
\rc&\Rc&\rc&\Rc&\rc&\Rc&&&&\\
 &\rc&\Rc&\rc&\Rc&\rc&\Rc&&&\\
 &  &\rc&\Rc&\rc&\Rc&\rc&\Rc&&\\
 & &&\rc&\Rc&\rc&\Rc&\rc&\Rc&&\\
& & &&\rc&\Rc&\rc&\Rc&\rc&\Rc&&\\
&& & &&\rc&&\rc&&\rc&&&\\
\end{array}}_{R}
$$
where the central matrix can be expressed as
$$
\underbrace{
\begin{array}{c@{\hspace{1mm}}c@{\hspace{1mm}}c@{\hspace{1mm}}c@{\hspace{1mm}}c@{\hspace{1mm}}c@{\hspace{1mm}}c@{\hspace{1mm}}c@{\hspace{1mm}}c@{\hspace{1mm}}c@{\hspace{1mm}}c@{\hspace{1mm}}c@{\hspace{1mm}}c@{\hspace{1mm}}c@{\hspace{1mm}}c@{\hspace{1mm}}c@{\hspace{1mm}}c@{\hspace{1mm}}c}
& & & &\cdot	& \cdot  &\cdot & \times& \times & \times& \times&\times & \times&  \\
&&&	&\cdot  &\cdot  & \cdot& \times& \times & \times& \times&\times & \times&   \\
&&&	&\cdot & \cdot& \cdot  & \times& \times & \times& \times&\times & \times& \\
&&	& & & \phantom{\Rc}&\Rc &\times &   & & &   \\
&&	& & & & \rc&\Rc & \times & & &   \\
&&	& & & & & \rc&\Rc &\times & &   \\
&&	& & & & &   &\rc &\Rc &\times   \\
&&	& & & & &   & &\rc &\Rc& \times  \\
&&	& & & & & &    & &\rc & &\times \\
\end{array}}_{Q+TZ^H}=\underbrace{\begin{array}{c@{\hspace{1mm}}c@{\hspace{1mm}}c@{\hspace{1mm}}c@{\hspace{1mm}}c@{\hspace{1mm}}c@{\hspace{1mm}}c@{\hspace{1mm}}c@{\hspace{1mm}}c@{\hspace{1mm}}c@{\hspace{1mm}}c@{\hspace{1mm}}c@{\hspace{1mm}}c@{\hspace{1mm}}c@{\hspace{1mm}}c@{\hspace{1mm}}c@{\hspace{1mm}}c@{\hspace{1mm}}c}
			&	&\times & & & & &  & &  \\
			&	&  &\times & &&   & & & &   \\
			&	& & &\times & &  & & & & & \\
			& & & \phantom{\Rc}&\Rc &\times &   & & &   \\
			& & & & \rc&\Rc & \times & & &   \\
			& & & & & \rc&\Rc &\times & &   \\
			& & & & &   &\rc &\Rc &\times   \\
			& & & & &   & &\rc &\Rc& \times  \\
			& & & & & &    & &\rc & &\times \\
			\end{array}}_{Q}\;\; 
		+ \underbrace{ \begin{array}{c@{\hspace{1mm}}c@{\hspace{1mm}}c@{\hspace{1mm}}c@{\hspace{1mm}}c@{\hspace{1mm}}c@{\hspace{1mm}}c@{\hspace{1mm}}c@{\hspace{1mm}}c@{\hspace{1mm}}c@{\hspace{1mm}}c@{\hspace{1mm}}c@{\hspace{1mm}}c@{\hspace{1mm}}c@{\hspace{1mm}}c@{\hspace{1mm}}c@{\hspace{1mm}}c@{\hspace{1mm}}c}
			\times& & & \times&\times& \times&  \times&\times &\times  \\
			&\times & & \times&\times& \times&  \times&\times &\times  \\
			& &\times & \times&\times& \times&  \times&\times &\times  \\
			&   &   &  &  &   &    &   &   \\
			&   &   &  &  &   &    &   &   \\
			&   &   &  &  &   &    &   &   \\
			&   &   &  &  &   &    &   &   \\
			&   &   &  &  &   &    &   &   \\
			&   &   &  &  &   &    &   &   \\
			\end{array}}_{TZ^H}
$$
and the $\cdot$ represent zeros. These zeros are obtained summing the contribution of the $k\times k$ principal blocks of $Q$ and of $TZ^H$ which sums up to zero. 
We have used the fact that 
$$
Q=\left[ \begin{array}{c|c}
I_k &\\ \hline
& \hat Q
\end{array}\right] D= {\cal G}_{k+1}\cdots {\cal G}_{N-1} \hat D,
$$ 
wherte ${\cal G}_i$ are Givens matrices acting on rows $i, i+1$ and $\hat D$ is a unitary diagonal matrix.  Furthermore,
in the lower left corner of the Schur parametrization of  $L$ we have trivial Givens rotations since  $ X(n+1:N, :)=-I_k$. The description of the bulge chasing algorithm in Section~\ref{sec:bulge-chasing} will make it clear why this structure is not  modified. 

Givens transformations can also interact with each other by means of the {\it fusion} or the {\it turnover} operations
(see~\cite{Raf_book}, pp.112-115).  The fusion operation will be depicted as follows:
\begin{equation*}
\begin{array}{cccc}
& \Rc\ca & \Rc  \\[-0.05cm]  
& \rc & \rc & \\[-0.05cm] 
\end{array}
\mbox{  resulting in  }
\begin{array}{cccc}
& \Rc   \\[-0.05cm]  
&\rc  \\[-0.05cm] 
\end{array},
\end{equation*}
and consists of the concatenation of two Givens transformations acting on the same rows.
The turnover operation allows to rearrange the order of some Givens trans\-for\-ma\-tions (see~\cite{Raf_book}).  

Graphically we will depict this rearrangement of Givens transformations as follows:
\begin{equation*}
\begin{array}{ccccc}
\Rc & \stb&\Rc  \\[-0.05cm]  
\rc & \Rc &\rc & \\[-0.05cm]
& \rc &     &  \\[-0.05cm] 
\end{array}\quad
\to
\quad
\begin{array}{ccccc}
&     & \Rc &  \\[-0.05cm]  
& \Rc & \rc & \Rc & \\[-0.05cm]
& \rc & & \rc &  \\[-0.05cm] 
\end{array}\qquad
\mbox{  or }
\quad
\begin{array}{cccccc}
&   & \Rc &  \\[-0.05cm]  
& \Rc & \rc & \Rc & \\[-0.05cm]
& \rc &$\rotatebox[origin=m]{180}{$\curvearrowleft$}$& \rc &  \\[-0.05cm] 
\end{array}
\quad
\to
\quad
\begin{array}{ccccc}
\Rc &  &\Rc  \\[-0.05cm]  
\rc & \Rc &\rc & \\[-0.05cm]
& \rc &     &  \\[-0.05cm] 
\end{array}.
\end{equation*}
$$
\begin{array}{ccccc}
\Rc & \slb&\Rc  \\[-0.05cm]  
\rc & \Rc &\rc & \\[-0.05cm]
& \rc &     &  \\[-0.05cm] 
\end{array}\quad
\to
\quad
\begin{array}{ccccc}
&     & \Rc &  \\[-0.05cm]  
& \Rc & \rc & \Rc & \\[-0.05cm]
& \rc & & \rc &  \\[-0.05cm] 
\end{array}\qquad
\mbox{  or }
\quad
\begin{array}{cccccc}
&   & \Rc &  \\[-0.05cm]  
& \Rc & \rc & \Rc & \\[-0.05cm]
& \rc &$\rotatebox[origin=m]{180}{$\curvearrowright$}$& \rc &  \\[-0.05cm] 
\end{array}
\quad
\to
\quad
\begin{array}{ccccc}
\Rc &  &\Rc  \\[-0.05cm]  
\rc & \Rc &\rc & \\[-0.05cm]
& \rc &     &  \\[-0.05cm] 
\end{array}.
$$

When we apply a sequence of $r$  consecutive turnover operations we will use the same symbol surmounted by the number of turnovers such as  $\stbt{r}$.

Each fusion and  turnover has a constant cost since consists in the operations involving  $2\times 2$ or $3\times 3$ matrices. Note that while the fusion between two Givens rotations can result in a trivial rotation, this is never the case when we perform a turnover between three non-trivial rotations.

\subsection{Initialization and bulge chasing}\label{sec:bulge-chasing}

As observed in Remark~\ref{rem:zrec} we do not have to perform the Givens transformations on the rank $k$ part since the matrix $Z$ can be  recovered at the end of the QR process and the  matrix $T$ is not affected by the transformations which act on rows $k+1$ to $N$. As we will explain in Section~\ref{sec:back} we prefer to  store explicitly the vectors $Z$ rather then recovering them at the end of the process because in this way we are able to prove a tighter bound for the  backward error of the method. 

The implicit QR algorithm starts with the computation of the shift.
Using a Wilkinson shift strategy we need to reconstruct
 the $2\times 2$ lower-right hand corner of $\hat A$.
This can be done by  operating on the representation and it  requires $O(k)$ flops. Once the
shift $\mu$ is computed, we retrieve the first two components of  the first column of
$\hat A$, i.e., $\hat a_{11}, \hat a_{21}$ and we compute the $2\times 2$ Givens rotation $G_1$ such that
$$
G_1\left[\begin{array}{c}\hat a_{11}-\mu\\ \hat a_{21}\end{array}\right]=\left[\begin{array}{c}\times\\ 0\end{array}\right].
$$
Let ${\cal G}_1=G_1\oplus I_{N-2}$, we have that matrix ${\cal G}_1 \hat A {\cal G}_1^H$  becomes
  
$$
	\underbrace{\begin{array}{r@{\hspace{1mm}}c@{\hspace{1mm}}c@{\hspace{-1mm}}r@{\hspace{1mm}}c@{\hspace{1mm}}c@{\hspace{1mm}}c@{\hspace{1mm}}c@{\hspace{1mm}}c@{\hspace{1mm}}c@{\hspace{1mm}}c@{\hspace{1mm}}c@{\hspace{1mm}}c@{\hspace{1mm}}c@{\hspace{1mm}}c@{\hspace{1mm}}c@{\hspace{1mm}}c@{\hspace{1mm}}c}
		& & & {G}_1&  \\
	&&   &      \Rc &\stbt{3} &\Rc&  & \\
		&&  &     \rc &\Rc&\rc& \Rc & \\
		& &    &\Rc&\rc  &\Rc   &\rc  &\Rc \\
		& &\Rc&\rc&   \Rc  & \rc  & \Rc &\rc \\
		&\Rc&\rc&   \Rc  & \rc  & \Rc &\rc \\
		\Rc&\rc&   \Rc  & \rc  & \Rc &\rc \\
		\rc&   \Rc & \rc  & \Rc &\rc \\
		& \rc  & \Rc &\rc \\
		&   & \rc & \\
		\end{array}}_{{\cal G}_1L} \;\; \left(
	\underbrace{\begin{array}{c@{\hspace{1mm}}c@{\hspace{1mm}}c@{\hspace{1mm}}c@{\hspace{1mm}}c@{\hspace{1mm}}c@{\hspace{1mm}}c@{\hspace{1mm}}c@{\hspace{1mm}}c@{\hspace{1mm}}c@{\hspace{1mm}}c@{\hspace{1mm}}c@{\hspace{1mm}}c@{\hspace{1mm}}c@{\hspace{1mm}}c@{\hspace{1mm}}c@{\hspace{1mm}}c@{\hspace{1mm}}c}
		& & & & & & & & & &&&&\\
		 & & & &\cdot	& \cdot  &\cdot & \times& \times & \times& \times&\times & \times&  \\
&&&	&\cdot  &\cdot  & \cdot& \times& \times & \times& \times&\times & \times&   \\
&&&	&\cdot & \cdot& \cdot  & \times& \times & \times& \times&\times & \times& \\
                    &  &  &  &  & \phantom{\Rc} & \Rc & \times &        &        &        &  \\
		                    &  &  &  &  &               & \rc &  \Rc   & \times &        &        &  \\
		                    &  &  &  &  &               &     &  \rc   &  \Rc   & \times &        &  \\
		                    &  &  &  &  &               &     &        &  \rc   &  \Rc   & \times &  \\
		                    &  &  &  &  &               &     &        &        &  \rc   &  \Rc   & \times &  \\
		                    &  &  &  &  &               &     &        &        &        &  \rc   &        & \times &
	\end{array}}_{Q+TZ^H}\;\; 
	\right) \underbrace{
		\begin{array}{c@{\hspace{1mm}}c@{\hspace{1mm}}c@{\hspace{1mm}}c@{\hspace{1mm}}c@{\hspace{1mm}}c@{\hspace{1mm}}c@{\hspace{1mm}}c@{\hspace{1mm}}c@{\hspace{1mm}}l@{\hspace{1mm}}c@{\hspace{1mm}}c@{\hspace{1mm}}c@{\hspace{1mm}}c@{\hspace{1mm}}l@{\hspace{1mm}}l@{\hspace{1mm}}l@{\hspace{1mm}}}
		&&   &&&&&&& {G}_1^H&  \\
		&&\Rc&&&&&&&\Rc \\
		&\Rc&\rc&\Rc&&&&&&\rc\\
		\Rc&\rc&\Rc&\rc&\Rc&&&&&\\
		\rc&\Rc&\rc&\Rc&\rc&\Rc&&&&\\
		&\rc&\Rc&\rc&\Rc&\rc&\Rc&&&\\
		&  &\rc&\Rc&\rc&\Rc&\rc&\Rc&&\\
		& &&\rc&\Rc&\rc&\Rc&\rc&\Rc&&\\
		& & &&\rc&\Rc&\rc&\Rc&\rc&\Rc&&\\
		&& & &&\rc&&\rc&&\rc&&&\\
		\end{array}}_{R{\cal G}_1^H}
$$

Applying a series of $k$ turnovers operations we can pass $G_1$ through the ascending sequence of Givens transformations, and a new  Givens transformation $G_{k+1}$ acting on rows $k+1$ and $k+2$, will appear before the bracket, and then is fused with the first nontrivial rotation of $Q$.     
$$
\underbrace{\begin{array}{r@{\hspace{1mm}}c@{\hspace{1mm}}c@{\hspace{1mm}}c@{\hspace{1mm}}c@{\hspace{1mm}}c@{\hspace{1mm}}c@{\hspace{1mm}}c@{\hspace{1mm}}r@{\hspace{1mm}}c@{\hspace{1mm}}c@{\hspace{1mm}}c@{\hspace{1mm}}c@{\hspace{1mm}}c@{\hspace{1mm}}c@{\hspace{1mm}}c@{\hspace{1mm}}c@{\hspace{1mm}}c}
		& & &&&&&\\
			&&   &      &     &\Rc&  & \\
		&&  &      &\Rc&\rc& \Rc & \\
		& &    &\Rc&\rc  &\Rc   &\rc  &\Rc \\
		& &\Rc&\rc&   \Rc  & \rc  & \Rc &\rc & \\
		&\Rc&\rc&   \Rc  & \rc  & \Rc &\rc & &\\
		\Rc&\rc&   \Rc  & \rc  & \Rc &\rc \\
		\rc&   \Rc & \rc  & \Rc &\rc \\
		& \rc  & \Rc &\rc \\
		&   & \rc & \\
		\end{array}}_{\bar L} \;\; \left(
		\underbrace{\begin{array}{c@{\hspace{1mm}}c@{\hspace{1mm}}c@{\hspace{1mm}}c@{\hspace{1mm}}c@{\hspace{1mm}}c@{\hspace{1mm}}c@{\hspace{1mm}}c@{\hspace{1mm}}c@{\hspace{1mm}}c@{\hspace{1mm}}c@{\hspace{1mm}}c@{\hspace{1mm}}c@{\hspace{1mm}}c@{\hspace{1mm}}c@{\hspace{1mm}}c@{\hspace{1mm}}c@{\hspace{1mm}}c}
		& & & & & & & & & &&&&\\
		& & & &\cdot	& \cdot  &\cdot & \times& \times & \times& \times&\times & \times&  \\
&&&	&\cdot  &\cdot  & \cdot& \times& \times & \times& \times&\times & \times&   \\
&&&	&\cdot & \cdot& \cdot  & \times& \times & \times& \times&\times & \times& \\
		&  &  &  &  & \Rc\ca & \Rc & \times &        &        &        &  \\
		&  &  &  &  &    \rc           & \rc &  \Rc   & \times &        &        &  \\
		&  &  &  &  &               &     &  \rc   &  \Rc   & \times &        &  \\
		&  &  &  &  &               &     &        &  \rc   &  \Rc   & \times &  \\
		&  &  &  &  &               &     &        &        &  \rc   &  \Rc   & \times &  \\
		&  &  &  &  &               &     &        &        &        &  \rc   &        & \times &
		\end{array}}_{{\cal G}_{k+1}Q+TZ^H}\;\; 
	\right) \underbrace{
		\begin{array}{c@{\hspace{1mm}}c@{\hspace{1mm}}c@{\hspace{1mm}}c@{\hspace{-1mm}}c@{\hspace{-1mm}}l@{\hspace{1mm}}c@{\hspace{1mm}}c@{\hspace{1mm}}c@{\hspace{1mm}}l@{\hspace{1mm}}c@{\hspace{1mm}}c@{\hspace{1mm}}c@{\hspace{1mm}}c@{\hspace{1mm}}l@{\hspace{1mm}}l@{\hspace{1mm}}l@{\hspace{1mm}}}
	   	&   &   &        & {G}_1^H&  \\
	    &   &\Rc&\slbt{3}&\Rc&&&&\\
		&\Rc&\rc&\Rc&\rc&&&&&\\
		\Rc&\rc&\Rc&\rc&\Rc&&&&&\\
		\rc&\Rc&\rc&\Rc&\rc&\Rc&&&&\\
		&\rc&\Rc&\rc&\Rc&\rc&\Rc&&&\\
		&  &\rc&\Rc&\rc&\Rc&\rc&\Rc&&\\
		& &&\rc&\Rc&\rc&\Rc&\rc&\Rc&&\\
		& & &&\rc&\Rc&\rc&\Rc&\rc&\Rc&&\\
		&& & &&\rc&&\rc&&\rc&&&\\
		\end{array}}_{R{\cal G}_1^H}
$$

Similarly the Givens rotation $G_1^H$ on the right  is shifted through the sequence of Givens transformations representing $R$ and  applied to the columns of $Z^H$ and on the right of ${\cal G}_{k+1}  Q$. Then another turnover operation is applied giving 
$$
	\underbrace{\begin{array}{r@{\hspace{1mm}}c@{\hspace{1mm}}c@{\hspace{1mm}}c@{\hspace{1mm}}c@{\hspace{1mm}}c@{\hspace{1mm}}c@{\hspace{1mm}}l@{\hspace{1mm}}c@{\hspace{1mm}}c@{\hspace{1mm}}c@{\hspace{1mm}}c@{\hspace{1mm}}c@{\hspace{1mm}}c@{\hspace{1mm}}c@{\hspace{1mm}}c@{\hspace{1mm}}c@{\hspace{1mm}}c}
		&&   &      &     &\Rc&  & \\
		&&  &      &\Rc&\rc& \Rc & \\
		& &    &\Rc&\rc  &\Rc   &\rc  &\Rc \\
		& &\Rc&\rc&   \Rc  & \rc  & \Rc &\rc & \\
		&\Rc&\rc&   \Rc  & \rc  & \Rc &\rc &\Rc &\\
		\Rc&\rc&   \Rc  & \rc  & \Rc &\rc &\sltt{3} &\rc G_{k+2} & \\
		\rc&   \Rc & \rc  & \Rc &\rc & & & \\
		& \rc  & \Rc &\rc \\
		&   & \rc & \\
		\end{array}}_{\bar L G_{k+2}} \;\; \left(
		\underbrace{\begin{array}{c@{\hspace{1mm}}c@{\hspace{1mm}}c@{\hspace{1mm}}c@{\hspace{1mm}}c@{\hspace{1mm}}c@{\hspace{1mm}}c@{\hspace{1mm}}c@{\hspace{1mm}}c@{\hspace{1mm}}c@{\hspace{1mm}}c@{\hspace{1mm}}c@{\hspace{1mm}}c@{\hspace{1mm}}c@{\hspace{1mm}}c@{\hspace{1mm}}c@{\hspace{1mm}}c@{\hspace{1mm}}c}
		& & & &\cdot	& \cdot  &\cdot & \times& \times & \times& \times&\times & \times&  \\
&&&	&\cdot  &\cdot  & \cdot& \times& \times & \times& \times&\times & \times&   \\
&&&	&\cdot & \cdot& \cdot  & \times& \times & \times& \times&\times & \times& \\
&  &  &  &  & \phantom{\Rc} & \Rc & \times &        &        &        &  \\
		&  &  &  &  &               & \rc &  \Rc   & \times &        &        &  \\
		&  &  &  &  &               &     &  \rc   &  \Rc   & \times &        &  \\
		&  &  &  &  &               &     &        &  \rc   &  \Rc   & \times &  \\
		&  &  &  &  &               &     &        &        &  \rc   &  \Rc   & \times &  \\
		&  &  &  &  &               &     &        &        &        &  \rc   &        & \times &
		\end{array}}_{Q_2+TZ_2^H}\;\; 
	\right) \underbrace{
		\begin{array}{c@{\hspace{1mm}}c@{\hspace{1mm}}c@{\hspace{1mm}}c@{\hspace{1mm}}c@{\hspace{1mm}}c@{\hspace{1mm}}c@{\hspace{1mm}}c@{\hspace{1mm}}c@{\hspace{1mm}}l@{\hspace{1mm}}c@{\hspace{1mm}}c@{\hspace{1mm}}c@{\hspace{1mm}}c@{\hspace{1mm}}l@{\hspace{1mm}}l@{\hspace{1mm}}l@{\hspace{1mm}}}
		&&\Rc&&&&&&&\\
		&\Rc&\rc&\Rc&&&&&&\\
		\Rc&\rc&\Rc&\rc&\Rc&&&&&\\
		\rc&\Rc&\rc&\Rc&\rc&\Rc&&&&\\
		&\rc&\Rc&\rc&\Rc&\rc&\Rc&&&\\
		&  &\rc&\Rc&\rc&\Rc&\rc&\Rc&&\\
		& &&\rc&\Rc&\rc&\Rc&\rc&\Rc&&\\
		& & &&\rc&\Rc&\rc&\Rc&\rc&\Rc&&\\
		&& & &&\rc&&\rc&&\rc&&&\\
		\end{array}}_{ R_2}
$$
At the end of the initialization phase the Givens rotation $G_{k+2}$ on the right of $\bar L$ can be brought on the left giving rise to the bulge represented by a Givens rotation $G_2$ acting on rows 2 and 3, namely $\bar L G_{k+2}=G_2^H L_2$. We have
$$
	\underbrace{\begin{array}{r@{\hspace{1mm}}c@{\hspace{1mm}}c@{\hspace{1mm}}c@{\hspace{1mm}}c@{\hspace{1mm}}c@{\hspace{1mm}}c@{\hspace{1mm}}l@{\hspace{1mm}}c@{\hspace{1mm}}c@{\hspace{1mm}}c@{\hspace{1mm}}c@{\hspace{1mm}}c@{\hspace{1mm}}c@{\hspace{1mm}}c@{\hspace{1mm}}c@{\hspace{1mm}}c@{\hspace{1mm}}c}
		&&   &      &     &\Rc&  & \\
		G_2^H&\Rc&  &      &\Rc&\rc& \Rc & \\
		&\rc &    &\Rc&\rc  &\Rc   &\rc  &\Rc \\
		& &\Rc&\rc&   \Rc  & \rc  & \Rc &\rc & \\
		&\Rc&\rc&   \Rc  & \rc  & \Rc &\rc & &\\
    \Rc&\rc&   \Rc  & \rc  & \Rc &\rc & & & \\
		\rc&   \Rc & \rc  & \Rc &\rc & & & \\
		& \rc  & \Rc &\rc \\
		&   & \rc & \\
		\end{array}}_{G_2^H L_2} \;\; \left(
	\underbrace{\begin{array}{c@{\hspace{1mm}}c@{\hspace{1mm}}c@{\hspace{1mm}}c@{\hspace{1mm}}c@{\hspace{1mm}}c@{\hspace{1mm}}c@{\hspace{1mm}}c@{\hspace{1mm}}c@{\hspace{1mm}}c@{\hspace{1mm}}c@{\hspace{1mm}}c@{\hspace{1mm}}c@{\hspace{1mm}}c@{\hspace{1mm}}c@{\hspace{1mm}}c@{\hspace{1mm}}c@{\hspace{1mm}}c}
		& & & &\cdot	& \cdot  &\cdot & \times& \times & \times& \times&\times & \times&  \\
&&&	&\cdot  &\cdot  & \cdot& \times& \times & \times& \times&\times & \times&   \\
&&&	&\cdot & \cdot& \cdot  & \times& \times & \times& \times&\times & \times& \\
		&  &  &  &  & \phantom{\Rc} & \Rc & \times &        &        &        &  \\
		&  &  &  &  &               & \rc &  \Rc   & \times &        &        &  \\
		&  &  &  &  &               &     &  \rc   &  \Rc   & \times &        &  \\
		&  &  &  &  &               &     &        &  \rc   &  \Rc   & \times &  \\
		&  &  &  &  &               &     &        &        &  \rc   &  \Rc   & \times &  \\
		&  &  &  &  &               &     &        &        &        &  \rc   &        & \times &
		\end{array}}_{Q_2+TZ_2^H}\;\;
	\right) \underbrace{
		\begin{array}{c@{\hspace{1mm}}c@{\hspace{1mm}}c@{\hspace{1mm}}c@{\hspace{1mm}}c@{\hspace{1mm}}c@{\hspace{1mm}}c@{\hspace{1mm}}c@{\hspace{1mm}}c@{\hspace{1mm}}l@{\hspace{1mm}}c@{\hspace{1mm}}c@{\hspace{1mm}}c@{\hspace{1mm}}c@{\hspace{1mm}}l@{\hspace{1mm}}l@{\hspace{1mm}}l@{\hspace{1mm}}}
		&&\Rc&&&&&&&\\
		&\Rc&\rc&\Rc&&&&&&\\
		\Rc&\rc&\Rc&\rc&\Rc&&&&&\\
		\rc&\Rc&\rc&\Rc&\rc&\Rc&&&&\\
		&\rc&\Rc&\rc&\Rc&\rc&\Rc&&&\\
		&  &\rc&\Rc&\rc&\Rc&\rc&\Rc&&\\
		& &&\rc&\Rc&\rc&\Rc&\rc&\Rc&&\\
		& & &&\rc&\Rc&\rc&\Rc&\rc&\Rc&&\\
		&& & &&\rc&&\rc&&\rc&&&\\
		\end{array}}_{ R_2}
$$

 The bulge  needs to be chased down. 

At this point we have ${\cal G}_1 \hat A{\cal G}_1^H={\cal G}_2^H  L_2(Q_2+T Z_2^H) R_2$, where ${\cal G}_2=1\oplus G_2\oplus I_{N-3}$. Performing a similarity transformation to get rid of   ${\cal G}_2^H$, we have that the  matrix $G_2^H$ on the right can be brought to the left applying turnover operations. Repeating the same reasoning $n-1$ times, we have  finally to remove a Givens rotation acting on columns $n-1$ and $n$. 

$$
\underbrace{\begin{array}{r@{\hspace{1mm}}c@{\hspace{1mm}}c@{\hspace{1mm}}c@{\hspace{1mm}}c@{\hspace{1mm}}c@{\hspace{1mm}}c@{\hspace{1mm}}c@{\hspace{1mm}}r@{\hspace{1mm}}c@{\hspace{1mm}}c@{\hspace{1mm}}c@{\hspace{1mm}}c@{\hspace{1mm}}c@{\hspace{1mm}}c@{\hspace{1mm}}c@{\hspace{1mm}}c@{\hspace{1mm}}c}
			&&   &      &     &\Rc&  & \\
		&&  &      &\Rc&\rc& \Rc & \\
		& &    &\Rc&\rc  &\Rc   &\rc  &\Rc \\
		& &\Rc&\rc&   \Rc  & \rc  & \Rc &\rc & \\
		&\Rc&\rc&   \Rc  & \rc  & \Rc &\rc & &\\
		\Rc&\rc&   \Rc  & \rc  & \Rc &\rc \\
		\rc&   \Rc & \rc  & \Rc &\rc \\
		& \rc  & \Rc &\rc \\
		&   & \rc & \\
		\end{array}}_{L_{n-1}} \;\; \left(
	\underbrace{\begin{array}{c@{\hspace{1mm}}c@{\hspace{1mm}}c@{\hspace{1mm}}c@{\hspace{1mm}}c@{\hspace{1mm}}c@{\hspace{1mm}}c@{\hspace{1mm}}c@{\hspace{1mm}}c@{\hspace{1mm}}c@{\hspace{1mm}}c@{\hspace{1mm}}c@{\hspace{1mm}}c@{\hspace{1mm}}c@{\hspace{1mm}}c@{\hspace{1mm}}c@{\hspace{1mm}}c@{\hspace{1mm}}c}
		& & & &\cdot	& \cdot  &\cdot & \times& \times & \times& \times&\times & \times&  \\
&&&	&\cdot  &\cdot  & \cdot& \times& \times & \times& \times&\times & \times&   \\
&&&	&\cdot & \cdot& \cdot  & \times& \times & \times& \times&\times & \times& \\
&  &  &  &  & \phantom{\Rc} & \Rc & \times &        &        &        &  \\
		&  &  &  &  &               & \rc &  \Rc   & \times &        &        &  \\
		&  &  &  &  &               &     &  \rc   &  \Rc   & \times &        &  \\
		&  &  &  &  &               &     &        &  \rc   &  \Rc   & \times &  \\
		&  &  &  &  &               &     &        &        &  \rc   &  \Rc   & \times &  \\
		&  &  &  &  &               &     &        &        &        &  \rc   &        & \times &
		\end{array}}_{Q_{n-1}+TZ_{n-1}^H}\;\;
	\right) \underbrace{
		\begin{array}{c@{\hspace{1mm}}c@{\hspace{1mm}}c@{\hspace{1mm}}c@{\hspace{1mm}}c@{\hspace{1mm}}c@{\hspace{1mm}}c@{\hspace{1mm}}c@{\hspace{1mm}}c@{\hspace{1mm}}l@{\hspace{1mm}}c@{\hspace{1mm}}
			c@{\hspace{1mm}}
			c@{\hspace{1mm}}c@{\hspace{1mm}}l@{\hspace{1mm}}l@{\hspace{1mm}}l@{\hspace{1mm}}
		}
			                    & \Rc & \rc & \Rc &     &     &     &     &          &  \\
			        \Rc         & \rc & \Rc & \rc & \Rc &     &     &     &          &  \\
			        \rc         & \Rc & \rc & \Rc & \rc & \Rc &     &     &          & G_{n-1}^H &  \\
			                    & \rc & \Rc & \rc & \Rc & \rc & \Rc &     & \slbt{3} & \Rc       &  \\
			                    &     & \rc & \Rc & \rc & \Rc & \rc & \Rc &          & \rc       &  &  \\
			                    &     &     & \rc & \Rc & \rc & \Rc & \rc &   \Rc    &           &  \\
			                    &     &     &     & \rc & \Rc & \rc & \Rc &   \rc    & \Rc       &  &  \\
			                    &     &     &     &     & \rc &     & \rc &          & \rc       &  &  &
		\end{array}}_{R_{n-1}{\cal G}_{n-1}^H}
$$
With the application of $k$ turnover operations, we get that $R_{n-1}{\cal G}_{n-1}^H={\cal G}_{n+k-1} R_n$, where ${\cal G}_{n+k-1}=I_{N-2}\oplus G_{N-1}$.  The Givens rotation $G_{N-1}$ acts on the last two columns and  will modify the last two columns of $Z_{n-1}^H$ and then fuses with matrix $Q_{n-1}$. At this point the Hessenberg structure is restored, and the implicit step ends. 

The graphical representation of the algorithm corresponds to the following updating of the matrices involved in the representation for suitable $P, S, V$ 

$$L^{(1)}=P^H L S,\quad Q^{(1)}
=S^HQV, \quad T^{(1)}=S^HT,
$$
$$Z^{(1)}=V^H Z,\mbox{ and } R^{(1)}=V^HRP.$$

In particular  in $P$ are gathered the $n-1$ rotations needed to restore the Hessenberg structure of  the full matrix, so that there are no operations involving the last $k$ rows of $L$, meaning that we can assume  $P_2=P(n+1:N, n+1:N)=I_k$.  $S$ is the product of the Givens rotations that have shifted through the factor $L$ when turnover operations are performed, and similarly $V$ is the product of the Givens matrices shifted through $R$ from the right. 

To show that this corresponds actually to a QR step it is sufficient to verify that we are under the hypothesis of Theorem~\ref{teo:rep1}, i.e., that $L^{(1)}$ and $R^{(1)}$ are unitary $k$-Hessenberg matrices, $T^{(1)}$ is still of the form $T^{(1)} =[T_k^T, 0]^T$ and that $Z^{(1)}$ has the structure described in point 2 of Theorem~\ref{teo:rep1}.  
From the description of the algorithm we can see that the matrices $S$ and $V$ are block diagonal with the leading block of size $k$ equal to the identity matrix since  the turnover operations shift down of  $k$ rows the rotations acting on rows and columns of $L$ and $R$ respectively.  We note that at the end of the chasing steps the $k$-Hessenberg structure of  $L^{(1)}$ and $R^{(1)}$ is restored, and $T^{(1)}=[T_k^H, 0]^T$ because $S(1:k, 1:k)=I_k$. Moreover, $L^{(1)}$ is still proper since the turnover operations cannot introduce trivial rotations. Matrix $Q^{(1)}$ is still block-diagonal  with the leading block $k\times k$ unitary diagonal, and the tailing block with Hessenberg structure. 
For $Z^{(1)}=V^HZ$ we need to prove that $Z^{(1)H}=L^{(1)}(n+1:N, :)Q^{(1)}$. From~\eqref{P}, and observing that $P_2=I_k$ we have  $L^{(1)}(n+1:N, :)= L(n+1:N, :)S$. 
Substituting we get 
\begin{eqnarray}
\label{eq:zespl}
Z^{(1)H}&=Z^H V= L(n+1:N, :)QV=  L(n+1:N, :)S S^HQV\\ \nonumber
&=L^{(1)}(n+1:N, :)Q^{(1)}
\end{eqnarray}
as required. 
To apply the implicit Q-Theorem we need to observe  that the first column of $A$ is only affected by the first rotation during the initialization step and is never changed after that.
%
%
\subsection{Computational cost}

The reduction of a generic  matrix to Hessenberg form requires in the general case $O(n^3)$ flops,  but as we observed in the introduction, in special cases the reduction can be achieved with $O(n^2k)$ operations. In~\cite{BDCG_TR}  is proposed  an $O(n^2k)$ algorithm to reduce a unitary-plus-low-rank matrix to Hessenberg form  and obtain directly the $LFR$ factorization suitable as starting point of the QR method we just described in this paper. The algorithm in~\cite{BDCG_TR} can be applied when
\begin{enumerate}
\item $A=D + U V^H$,   $U, V\in \mathbb C^{n\times k}$, and
$D$ is unitary  block  diagonal  with block size $k<n$.
\item $A=H + [I_k, 0]^T Z^H$,   $Z\in \mathbb C^{n\times k}$, and
$H$ is unitary  block  upper Hessenberg  with block size $k<n$;
\item $A=G + [I_k, 0]^T Z^H$,   $Z\in \mathbb C^{n\times k}$, and
$G$ is unitary  block CMV   with block size $k<n$;
\end{enumerate}
These three cases cover the most interesting structures of low-rank perturbation of unitary matrices. In the general case of unitary matrices, where the spectral factorization of the unitary part is not known, in general we cannot expect to recover the eigenvalues  even of the unitary part in $o(n^3)$. 

 Unitary matrices are always diagonalizable, so we fall in case (1) if we know the eigen-decomposition of the unitary part.   Block companion  matrices belong to case (2), and applying the algorithm in~\cite{BDCG_TR} to reduce them in Hessenberg form we get directly the factored representation.

We assume hence that $A$ is already in Hessenberg form and we know that the embedding preserves this structure. If we are not in cases (1)-(3) it is necessary to compute the matrices required 
for embedding  $A$ in $\hat A$ which  can be performed using  $O(n^2k)$ operations.
Similarly the cost of the  representation is $O(n^2k)$ operations, since we need to compute $O(nk)$ Givens rotations and apply them to $N\times N$ matrices.

The key ingredients of the algorithm are turnover or fusion operations. Each of such operations can be performed with a bounded number of operations since they involve only $2\times 2$ or $3\times 3$ matrices. Each QR step consists of an inizialization phase, requiring $3k+1$ turnovers and a fusion as well as the updating of two rows of the $n\times k$ matrix $Z$ in the case we choose to update it at each step.
Each of the remaining $n-2$ chasing steps consists of $2k+1$ turnovers and the possible update of $Z$.
In the final step we have $k$ turnovers and a fusion with the last Givens rotation in $Q$.
 Overall the cost of an implicit QR step is $O(nk)$, and assuming as usual that deflation happens after a constant number of steps, we get an overall cost of $O(n^2k)$ arithmetic operations for retrieving all the eigenvalues. Comparing  the cost of this algorithm with the cost of the unstructured QR method,   which requires $O(n^2)$ flops per iteration and then  a total cost of $O(n^3)$, shows the advantage of using this approach.

\subsection{Deflation}

Deflation  techniques  are based on equation~\eqref{adefl} which shows that the possibility of  performing deflation can be recognized by direct inspection on the representation without re\-con\-struc\-ting the matrix $A$. In practice it is equivalent to check the subdiagonal entries of the factor $Q$.

\begin{lemma}\label{lem:defcrit}
Assume that the QR iteration applied to the matrix $\hat A$ is convergent to an upper triangular matrix. Denote by $\hat A^{(s)}=L^{(s)}(Q^{(s)}+TZ^{(s)H})R^{(s)}$ the matrix obtained after $s$ steps of the  QR  algorithm. Then, for $i=1, \ldots, n-1$, we have $\lim_{s\to \infty} q^{(s)}_{i+k+1, i+k}= 0$, and moreover, for  any prescribed tolerance $\tau$, if $s$ is such that $|q_{i+k+1, i+k}^{(s)}|<\tau K$,  then $|a_{i+1,i}^{(s)}|<\tau$.
\end{lemma}
\begin{proof}
From relation~\eqref{adefl}	we have that 
$$
|q^{(s)}_{i+k+1, i+k}| =|a^{(s)}_{i+1, i}|\,\frac{\left|\bar l ^{(s)}_{i+1, i+k+1}\right|}{|r^{(s)}_{i+k, i}|}, 
$$
where $a_{i,j}^{(s)}$  are the entries of the matrix $A^{(s)}$  defined according to Theorem \ref{teo-qr}.
Using Corollary~\ref{cor_prod} and the convergence of the QR algorithm we have
 $\lim_{s\to \infty} q^{(s)}_{i+k+1, i+k}=0$. 
	
From Corollary~\ref{cor_prod} we have $|r^{(s)}_{i+k, i}|\le 1$ and $\left|\bar l^{(s)}_{i+1, i+k+1}\right|\ge K$. Hence 
$$
|a^{(s)}_{i+1, i}|=|q^{(s)}_{i+k+1, i+k}|\,\frac{|r^{(s)}_{i+k, i}|}{\left|\bar l^{(s)}_{i+1, i+k+1}\right|}\le \tau.
$$
\qed
\end{proof}

\begin{remark}
	Lemma~\ref{lem:defcrit} suggests to use as deflation criteria the condition \\ $|q_{i+k+1, i+k}^{(s)}|<\varepsilon K$, where $\varepsilon$ is the machine precision. The value of $K$ as described in Corollary~\ref{cor_prod}  can be computed as $1/|\det(T_k)|$ for the upper triangular matrix $T_k$ given in the proof of Theorem~\ref{teo:rep}. Note that, as we represent the matrix $Q$  in terms of Givens rotations $\begin{bmatrix} c &s\\-s &\bar c\end{bmatrix}$, with $s\in \mathbb{R}, s\ge 0$, we can simply check when a sine value is smaller than $\varepsilon K$. When this condition is satisfied we replace the corresponding Givens transformation  with the $2\times 2$ unitary diagonal  matrix
	$D_j=\left[\begin{array}{cc}c/|c|&\\  &\bar c/|c| \end{array}\right]$.  In~\cite{MV14} it is shown, in a more general setting, that the eigenvalues of the matrix obtained by replacing a Givens transformation by a $2\times 2$ identity matrix are accurate approximations of the original ones when the Givens rotation is close to the identity. This is a consequence of the Bauer-Fike Theorem. Applying the same  idea to our framework it is immediate to see that the absolute error introduced in a single eigenvalue is bounded by $\kappa_2\|G_j-D_j\|_2$ where  $\kappa_2$ is the condition number of the eigenvector matrix (assuming to work with diagonalizable matrices)  and $G_j=\left[\begin{array}{cc}c&s\\ -s &\bar c \end{array}\right]$, $j=i+k+1$ is the Givens rotation such that $|q_{i+k+1, i+k}|=s\le \varepsilon K$. Moreover we can bound $\|G_j-D_j\|_2$ by $\sqrt{2} \varepsilon  K$. 
\end{remark}

\section{Backward error analysis}\label{sec:back}

In this section we bound the backward error of the shifted QR algorithm presented in Section~\ref{sec:alg}.  The  algorithm basically can be described by the following steps:
\begin{enumerate}
\item {\em Preliminary phase}:   the input unitary-plus-low-rank Hessenberg 
matrix  $A$  is embedded into a larger  Hessenberg matrix $\hat A$; 
\item {\em Initialization phase}: a  compressed $LFR$-type representation of $\hat A$  is  computed, $\hat A=L^{(0)}F^{(0)}R^{(0)}$;
\item {\em Iterative phase}: at each step $h$, given the representation of the matrix $\hat A^{(h)}=L^{(h)}F^{(h)}R^{(h)}$ we perform a  shifted QR iteration with the proposed algorithm, and  return $A^{(h+1)}=L^{(h+1)}F^{(h+1)}R^{(h+1)}$.
\end{enumerate}

As suggested in the introduction the initialization phase can be  determined  in several different ways depending on the additional
features of the input matrix.  In this section to be consistent with the   approach   pursued  in the previous sections
we  only consider the case where  the representation is found by a sequence of QR factorizations.

Concerning  the preliminary phase we notice that as in the proof of Theorem~\ref{theo:embedding} we can first  compute  the economy size QR decomposition of the full rank matrix $Y$. If we set $Y=QR$  and then rename the components of $A=U+XY^H$ as follows
$U\leftarrow U$, $X\leftarrow XR^H$   and 
$Y\leftarrow Q$  we find that $ Y^HY=I_k$. In this way  the embedding  is performed  at a  negligible  cost  by introducing a small error
perturbation of order $\tilde \gamma_k \|A\|_2$ where $\tilde \gamma_k=c k\varepsilon /(1-c k\varepsilon)$~\cite{Higham_book} and 
$c$ and $\varepsilon$  denote a small integer constant and the
machine precision, respectively. For the sake of simplicity it is assumed that   $\|X \|_2=\|X \|_2 \|Y\|_2\approx \|A\|_2,$ and we refer to the notation used in~\cite{Higham_book}.

\subsection{Backward stability of the representation}

In this section we prove that our representation is backward stable. The ingredients of the representation essentially are: the $k$-lower Hessenberg matrix $L$, the upper Hessenberg matrix $Q$, the $k\times k$ upper triangular matrix $T$, matrix $Z^H$ and the $k$-upper Hessenberg matrix $R$. 
In particular, given un upper Hessenberg matrix $\hat A=\hat U+\hat X\hat Y^H$ we would like to show that the exact representation of $\hat A=L(Q+TZ^H)R$ differs from the computed one $\tilde A=\tilde L(\tilde Q+\tilde T\tilde Z^H)\tilde R$ by an amount proportional to $\|\hat A\|_2\approx \|A\|_2$ and to the machine precision $\varepsilon$.

The computation proceeds by the following steps.  
\begin{itemize}
	\item Computation of $T$ and $L$. We note that $L$ and $T$, $T=[T_k^T;0]^T$ in exact arithmetics are respectively  the $Q$ and the $R$ factors of the QR factorization of matrix $\hat X$. From Theorem 19.4 of~\cite{Higham_book} and consequent  considerations there exists a perturbation $\Delta_X$ such that 
	 \begin{equation}\label{xdx}
(\hat X+\Delta_X)= \tilde L \tilde T,
	 \end{equation} where
	$\|\Delta_{X}(:,j)\|_2\le  \sqrt{k} \,\tilde\gamma_{Nk}\|\hat X(:,j)\|_2$,  $\tilde L=L+\Delta_L$ , $\|\Delta_L(:,j)\|_2\le \tilde \gamma_{Nk}.$ 
       \item  Computation of $Q$. The computed matrix $\tilde Q$ is obtained starting from matrix $\tilde B=\tilde L^H (\hat U+\Delta_U)\doteq L^H\hat U+{\Delta_L}^H  \hat U+L^H \Delta_U$, and  $\|\Delta_U\|_2\lesssim \tilde \gamma_{N^{2}}$, which derives from the backward analysis of the product of two unitary factors.
        The computed factor is  $$\tilde Q=\left[\begin{array}{c|c} I_k & \\ \hline  &\tilde Q_1\end{array} \right],$$ where $Q_1$ is obtained applying the QR factorization to $\tilde B(k+1:N, 1:n)$.
 Reasoning similarly as done before we have
$$
\tilde B(k+1:N, 1:n)+\Delta_B^{(1)}=\tilde Q_1 \tilde R_1,
$$
where $\| \Delta_{B(:, j)}^{(1)}\|_2\le \tilde \gamma_{nk}$, and $\tilde Q=Q+\Delta_Q$ with $\|\Delta_Q(:,j)\|_2\le \tilde \gamma_{nk}.$ 
 \item Computation of $R$. Similarly the computed $\tilde R$ is such that 
 $
 \tilde B+ \Delta_B^{(2)}=\tilde Q \tilde R,
 $
 where $||\Delta_B^{(1)}\|_2$ is small and $\tilde R=R+\Delta_R$  with $\|\Delta_R(:,j)\|_2\le \tilde \gamma_{Nk}.$ 
\item Computation of $Z^H$. As we have seen we can perform QR iterations without computing $Z$ explicitly because we can retrieve the matrix $Z$ at convergence using formula~\ref{form:Z}.
 However for the stability of the all process we prefer  to work with explicit $Z$, updating it at each step. This will affect the computational cost of a factor $O(nk)$ at each step but results in a more stable method (see Figure~\ref{fig:bsZ}  and the  related  discussion). 
 
 Note that in exact arithmetic $Z=R Y$. Since the product by a small perturbation of a unitary matrix is backward stable, we get 
 $$
 \tilde Z=\tilde R(Y+\Delta Y), \quad \mbox{ where }  \|\Delta Y\|_2 \lesssim \tilde \gamma_{N^{2}}. 
 $$

\end{itemize}
We can conclude that
\begin{eqnarray}
\tilde A&=&\tilde L(\tilde Q+\tilde T \tilde Z^H)\tilde R=\tilde L(\tilde B+\Delta_B^{(2)} )+\tilde L \tilde T\tilde Z^H\tilde R=\\
&=& \tilde L(\tilde L^H (\hat U+\Delta_U)+\Delta_B^{(2)})+ (\hat X+\Delta_X)(\hat Y+\Delta_Y)^H \tilde R^H \tilde R.
\end{eqnarray}
Since $\tilde L \tilde L^H=I+\Delta_1$, and $\tilde R^H \tilde R=I+\Delta_2$ with $\|\Delta_1\|_2$  and $\|\Delta_2\|_2$ bounded by a constant times $ \tilde \gamma_{N^{2}}$, we have
$$
\tilde A\doteq \hat U+\hat X\hat Y^H+E= \hat A+E
$$
where
$
E=\Delta_1 \hat U+\Delta_U+L \Delta_B^{(2)}+\hat X\hat Y^H \Delta_2+\Delta_X \hat Y^H+\hat X \Delta_Y^H,
$ 
and $\|E \|_2$ is bounded by a constant times $\|\hat X \|_2 \|\hat Y\|_2\approx \|A\|_2.$

\subsection{Backward stability of the implicit  QR steps}

Matrix $A^{(1)}$, computed by a QR step applied to matrix $\hat A$, is such that $A^{(1)}=P^H\hat AP$ with $P$ unitary.
In this section we want to show that there exists a perturbation matrix $E_A$ such that the computed matrix $\tilde A^{(1)}=P^H(\hat A+E_A)P$ where $\|E_A\|_2$ is proportional to $\|A\|_2$ and to the machine precision $\varepsilon$.

Let $\hat A=L(Q+TZ^H)R$ be the representation of $\hat A$ in floating point arithmetic (note that for not  overloading the notation we dropped the superscripts).  Similarly the new  representation of the matrix  $A^{(1)}=P^H \hat A P=P^H\,L(Q+TZ^H)R \,P$
is $A^{(1)}=L^{(1)}(Q^{(1)}+T^{(1)}Z^{(1)H})R^{(1)}$, where
\begin{eqnarray*}
L^{(1)}&=P^H L S,\quad Q^{(1)}
=S^HQV, \quad T^{(1)}=S^HT,\\
& Z^{(1)}=V^H Z,\mbox{ and } R^{(1)}=V^HRP,
\end{eqnarray*}
and the unitary matrices $S$ and $V$ are those originating from the turnovers on the Givens rotations composing $L$ and $R$.
\begin{theorem}\label{teobsf}
	After one step of the implicit QR method where the operations are performed in floating point arithmetic we have
	$$\tilde L^{(1)}=P^H(L+E_L)S,\quad \tilde Q^{(1)}
	=S^H(Q+E_Q)V, \quad T^{(1)}=S^HT,\quad
	$$
	$$ Z^{(1)}=V^H (Z+E_Z),\mbox{ and } R^{(1)}=V^H(R+E_R)P,$$
        where $\|E_L\|_2, \|E_Q\|_2, \|E_R\|_2$ are bounded by a small multiple of the machine precision $\varepsilon$,
        while $\|E_Z\|_2$ is a bounded by $\tilde \gamma_N$, where $\tilde \gamma_N=cN\varepsilon/(1-cN\varepsilon)$, and  $c$ is a small constant.  
\end{theorem}
\begin{proof}
  The backward analysis of the error in the unitary factors, $L, R$ and $Q$ is similar to the one performed in Theorem 7.1 in \cite{AMVW15}.
  To prove that the backward error in $T^{(1)}$ is zero we note that, because  of the structure of $S$ and of $T$, the product  $S^HT$ is in
  practice never computed since, $S^HT=T$.  The computation of $Z^{(1)}$
  is the result of the product of a  unitary factor and the rectangular matrix $Z$ whose columns are orthonormal. For this reason $\|E_Z\|_2$
  is bounded by   $\tilde \gamma_N\|Y\|_2=\tilde \gamma_N$.\qed
\end{proof}	

Summarizing we obtain the following result.

\begin{theorem}
  Let $\tilde A^{(1)}$ be the result computed in floating point arithmetic  of a QR step applied to the  matrix $\hat A$. Then there exists a perturbation
  $\Delta_A$ such that  $\tilde A^{(1)}=P^H(\hat A+\Delta_A)P$, and $\|\Delta_ A\|_2\le \tilde \gamma_{N^2}\|\hat A\|_2$, where $\tilde \gamma_{N^2}=cN^2\varepsilon/(1-cN^2\varepsilon)$, and  $c$ is a small constant. 
\end{theorem}

\begin{proof}
The proof follows easily by using the results proved in Theorem~\ref{teobsf} by  assembling together the contributions of each error in the factors of $A$.	\qed
\end{proof}

\section{Numerical experiments} \label{sec:numexp}
We tested our algorithm on several classes of matrices.
The purpose of the experimentation is 
to show that the algorithm is indeed backward stable as proved in Section~\ref{sec:back}  
and  to confirm that the computation of all the eigenvalues
by our method requires  $O(n^2k)$ operations as proved theoretically. 

The test suite consists of the  following:
\begin{itemize}
\item{Companion matrices associated with scalar polynomials whose roots are known} (see description in Table~\ref{tab:scal_pol}).
\item{Companion matrices associated with scalar polynomials whose roots are unknown}
  (see  description in Table~\ref{tab:scal_pol}).
\item{Random fellow matrices~\cite{Raf_book} with a prescribed norm.}
\item {Block companion matrices associated with matrix polynomials from the NLEVP collection~\cite{NLEVP}.}
\item{Random unitary plus rank-$k$ matrices}.
\item{Random unitary-diagonal-plus-rank-$k$ matrices}.
\item{Fiedler penta-diagonal companion matrices~\cite{Fi03,CM2}}. The associated polynomials are the scalar polynomials in Table~\ref{tab:scal_pol} and we have then associated the same reference number.
\end{itemize}

In~\cite{DDP14} the backward stability of polynomial rootfinding using  Fiedler  matrices different from the companion matrix is analyzed. The analysis shows that, when some coefficients of the monic polynomial are large,  it is preferable to use the standard  companion form. However, when the coefficients can be bounded by a moderate number, the algorithms using Fielder matrices are as backward stable  as the ones using the companion. The purpose of our experiments with Fiedler pentadiagonal matrices is not to suggest to use these matrices for polynomial rootfinding, but to show that the backward stability of the proposed method is not affected by a larger $k$ since for Fiedler pentadiagonal matrices we have $k=n/2$. 

\begin{table}
{\small
\begin{tabular}{|r|l|r|}
	\hline
	$\#$ & Description, Roots                                  & Degree \\ \hline
	1    & Wilkinson polynomial \quad                                                                                                                                     $1, 2, \ldots n$ &    $n$ \\
	2    & Scaled and shifted Wilkinson polynomial                                                       \qquad                                                      $-2.1,-1.9, \ldots, 1.7$ &    $n$ \\
	3    & Reverse Wilkinson polynomial                                                                  \qquad                                                       $ 1, 1/2, \ldots, 1/n$ &    $n$ \\
	4    & Prescribed Roots                                                                              \qquad                                          $2^{-m}, 2^{-(m-1)}, \ldots, 2^{m}$ & $2m+1$ \\
	5    & Prescribed roots shifted \qquad                                                                                                                       $2^{-m}-3, \ldots, 2^{m}-3$ & $2m+1$ \\
	6    & Chebyshev polynomial                                                                          \qquad                                       $\cos\left(\frac{(2j-1)\pi}{2n}\right)$ &    $n$ \\
	7    & $\sum_{i=0}^n x^i $                                                                           \qquad $\cos\left(\frac{2\pi j}{n+1}\right)+i\, \sin\left(\frac{2\pi j}{n+1}\right)$ &    $n$ \\
	\hline
	8   & Bernoulli polynomial                                                                          \qquad                                                                            -- &    $n$ \\
	9   & $p_1(z)=1+\left(m/(m+1)+(m+1)/m\right) z^m +z^{2m}$                                           \quad                                                                            -- &   $2m$ \\
	10   & $p_2(z)=\frac{1}{m} \left(\sum_{j=0}^{m-1}(m+j)z^j+(m+1)z^m+\sum_{j=0}^{m-1} z^{2m-j}\right)$ \quad                                                                            -- &   $2m$ \\
	11   & $p_3(z)=(1-\lambda)z^{m+1}-(\lambda-1)z^m +(\lambda+1) z+(1-\lambda)$ , $\lambda=0.999$       \quad                                                                           -- &  $m+1$ \\
	\hline 
\end{tabular}}
\caption{In the upper part of the table scalar polynomials  whose roots are known. These polynomials have also been tested in~\cite{CG,AMVW15}, at the bottom of the table polynomials with particular structures, tested also in~\cite{JT,BDG}.\label{tab:scal_pol}}
\end{table}
\begin{figure}
	
	\includegraphics[width=1.0\linewidth]{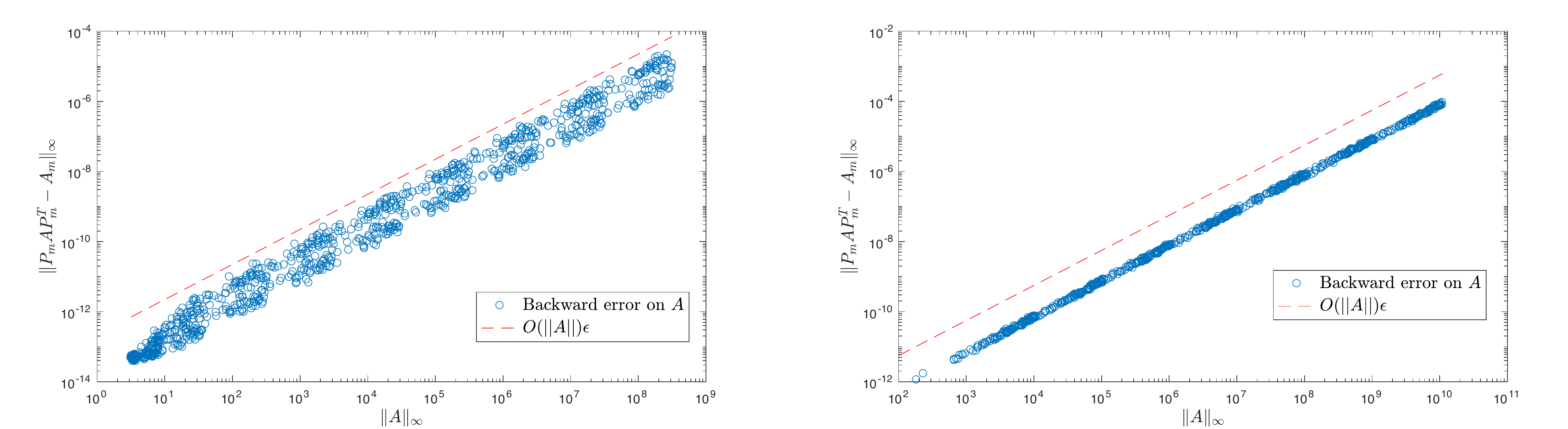}
	
	\caption{Absolute backward error in the computation of the eigenvalues respect to the norm of the matrix. On the left the results obtained from thousand random unitary-plus-rank-$5$ matrices of size $50$  that were generated as explained in Theorem~\ref{teo:rep1}. On the right the absolute backward error is plotted against the norm of the matrix for one thousand unitary diagonal-plus-rank-5 matrices of size 100. The dashed lines represent a reference line for the theoretical  backward stability.  \label{fig:bs}}
\end{figure}

\begin{figure}
	
	\includegraphics[width=0.8\linewidth]{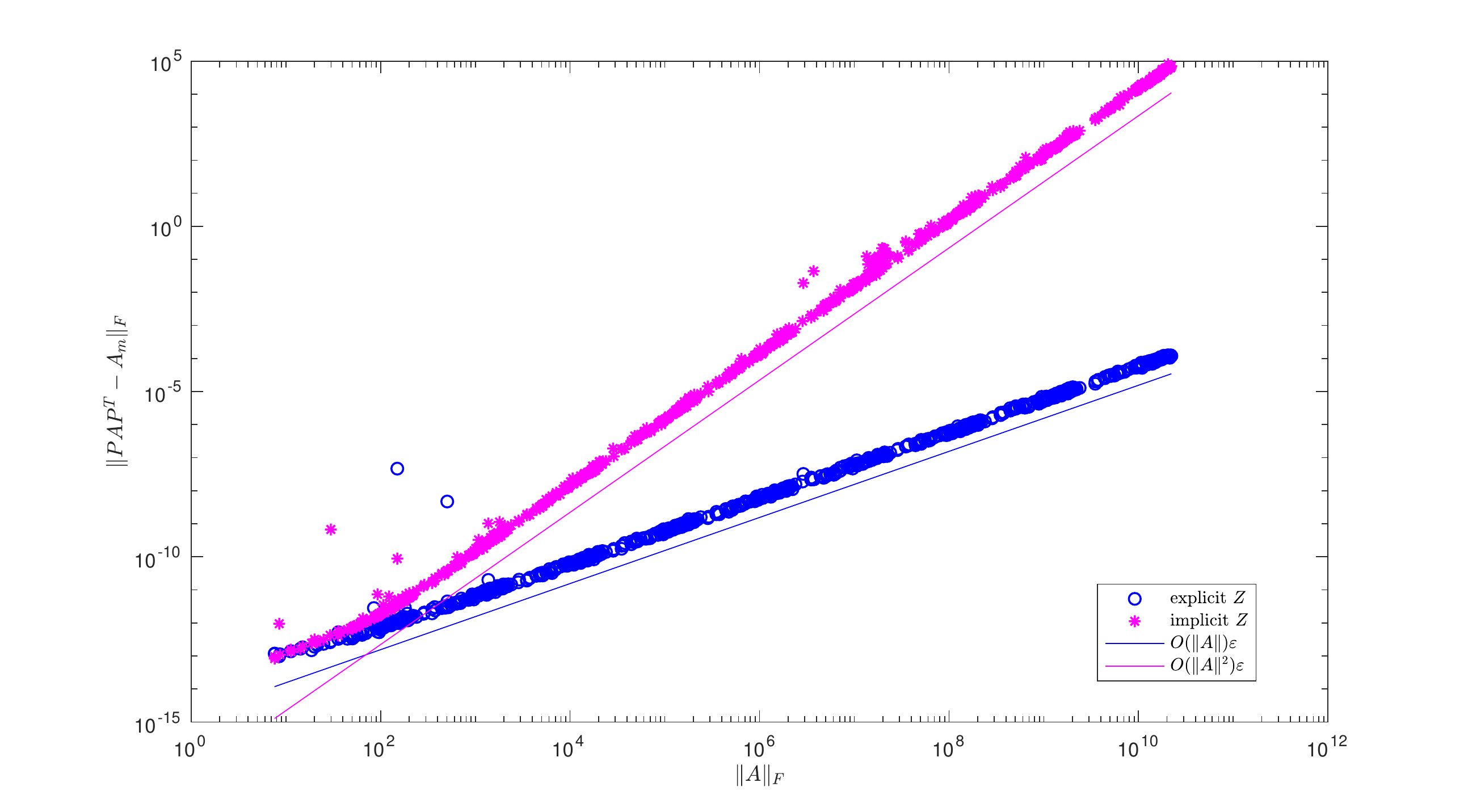}
	
	\caption{Comparison on the backward errors of the algorithm with an explicit and implicit $Z$. We plotted the absolute backward error respect to the norm of the matrix. The results are obtained from a thousand random matrix polynomials of degree $10$ where the coefficients are $5\times 5$ matrices, whose norms range from $1$ to $10^9$.  The solid lines represent a reference lines to show that in the case  $Z$ is explicitly computed  the absolute backward error behaves as $O(\|A\|)\varepsilon$ while, keeping $Z$ implicit, the backward error increases as  $O(\|A\|^2)\varepsilon$. \label{fig:bsZ}}
	\end{figure}

The algorithm is implemented in Matlab and the code is available upon request. 
In order to check the accuracy of the output we compare the computed approximations
with the actual eigenvalues of the matrix, in the case these are known. Otherwise we consider the values  returned  by the internal Matlab function {\tt eig} applied to the initial
matrix $A$ already in Hessenberg form and  with the balancing option on. 
Specifically,
we match the two lists of approximations and then find the average error as
$$
\mbox{fw}_{\mbox{err}} =\frac{1}{n}\sum_{j=1}^n err_j,
$$
where $err_j$ is the relative error in the computation of the $j$-th eigenvalue.
The eigenvalues are retrieved reconstructing the matrix at convergence and taking the diagonal entries. Note that unitary factor in $F$ has reduced to a unitary  diagonal matrix, plus the  the rank $k$ part confined in the first $k$ rows.

\begin{table}
	\begin{center}
		\begin{tabular}{lcccccccc}
			\hline
			\#                                    &     $n$      & $\|A\|_{\infty}$ & $\mbox{bw}_{\mbox{err}}(A)$ &   $\mbox{bw}_{\mbox{err}}(p)$    &     AMVW      &      BEGG      &     ZHSEQR        \\ \hline
			1                                     & $10       $  &   $ 1.93e+07$    &    $ 1.68e-15         $     & $ 6.31e-15                     $ & $ 5.12e-15  $ & $ 1.55e-15   $ &  $ 4.11e-16  $    \\
			1                                     & $ 15       $ &   $ 9.62e+12$    &    $ 1.00e-15         $     & $ 8.90e-15                     $ & $ 3.96e-15 $  & $ 3.45e-15   $ &  $ 1.33e-15 $     \\
			1                                     & $ 20       $ &   $ 2.28e+19$    &    $ 2.03e-15         $     & $ 5.28e-14                     $ & $ 1.04e-14 $  & $ 3.35e-01  $  &  $ 4.47e-15  $   \\
			2                                     & $ 20       $ &   $ 6.69e+02$    &     $ 1.55e-15       $      & $ 1.36e-14                     $ & $ 8.21e-16 $  & $ 2.17e-15   $ &  $ 1.09e-15  $    \\
			3                                     & $ 20       $ &    $1.03e+01$    &    $ 3.58e-15         $     & $ 8.08e-15                     $ & $ 2.23e-15 $  & $ 2.93e-14   $ &  $ 1.24e-15  $    \\
			4                                     & $ 20       $ &   $ 1.93e+14$    &    $ 1.55e-15         $     & $ 4.98e-14                     $ & $ 2.70e-15  $ & $ 4.32e-14   $ &  $ 4.44e-16  $    \\
			5                                     & $ 20       $ &    $1.33e+18$    &    $ 1.44e-15         $     & $ 4.41e-14                     $ & $ 2.72e-14  $ & $ 5.42e-01   $ &  $ 4.38e-15  $    \\
			6                                     & $ 20       $ &   $ 2.02e+01$    &    $ 1.63e-15         $     & $ 1.70e-14                     $ & $ 1.54e-15  $ & $ 1.79e-14   $ &  $ 2.52e-15  $    \\
			7                                     & $ 20       $ &    $6.32e+00$    &    $ 3.41e-15         $     & $ 1.81e-14                     $ & $ 2.00e-15  $ & $ 2.67e-14   $ &  $ 2.85e-15  $   \\ \hline
		    8                                   & $ 20       $ &   $ 6.76e+04$    &    $ 1.86e-15         $     & $ 2.50e-14                     $ & $ 2.17e-15  $ & $ 1.70e-14   $ &  $ 4.56e-15  $   \\
			9                                    & $ 40       $ &    $6.71e+00$    &    $ 7.98e-15         $     & $ 1.87e-13                     $ & $ 9.95e-17  $ & $ 7.07e-01   $ &  $ 3.52e+13  $  \\
			10                                    & $ 40       $ &    $1.14e+01$    &    $ 5.00e-15         $     & $ 3.10e-14                     $ & $ 4.99e-15  $ & $ 3.92e-14  $  &  $ 6.11e-15  $    \\
			10                                    & $ 20       $ &    $7.99e+00$    &    $ 2.89e-15         $     & $ 1.27e-14                     $ & $ 2.94e-15  $ & $ 1.96e-14  $  &  $ 4.85e-15  $    \\
			11                                    & $ 31       $ &   $ 2.83e+03$    &    $ 1.91e-15         $     & $ 4.64e-13                     $ & $ 4.89e-15  $ & $ 2.43e-13   $ &  $ 1.03e-14  $    \\ \hline
		\end{tabular}
	\end{center}\caption{Results on scalar companion matrices from Table~\ref{tab:scal_pol}. We see that the backward error is always very small,  the forward error (with is not shown) is  dependent on the conditioning of the problem. In the last three column we report~\cite{AMVW15} the backward error in terms of the polynomial coefficients of the unbalanced, single shifted version of the algorithm  AMVW~\cite{AMVW15}, of BEGG~\cite{BEGG}, and the LAPACK routine ZHSEQR.  \label{tab:sc_pol} }
\end{table}

To validate the results provided in Section~\ref{sec:back} we show the behavior of the backward error on the computed Schur form.   Let $P$ be the accumulated unitary similarity transformation obtained applying steps of the implicit QR algorithm as described in Section~\ref{sec:alg} to the augmented matrix $\hat A$.  Because the last $k$  rows of $\hat A$ are null according to Theorem \ref{teo-qr}, $P$  is block diagonal 
$$
P=\left[\begin{array}{cc}P_1 & \\ & P_2\end{array}\right],
$$ 
where $P_1\in \mathbb{C}^{n\times n}, P_2\in \mathbb{C}^{k\times k}$ are unitary matrices. We can   set $P_2=I_k$ since no rotations act on the last $k$ rows of the enlarged matrix.  Assume that $m$ is the number of iterations needed to reach convergence  and that  $\tilde A^{(m)}$ is the   matrix reconstructed from the computed factors $\tilde L_m$,  $\tilde F_m$, $\tilde R_m$ produced by performing $m$ steps of the implicit algorithm. Not to overload the notation we denote with the same symbol $\tilde A^{(m)}$ its $n\times n$  leading principal submatrix.  As in~\cite{CG} we consider as a measure of the backward stability the relative error 
\begin{equation}\label{be}
\mbox{bw}_{\mbox{err}}(A)=\frac{\| P_1^TA P_1-\tilde A^{(m)}\|_{\infty}}{\|A\|_{\infty}}.
\end{equation}

In order to compare the stability of our algorithm with that of algorithm tailored for polynomial rootfinding~\cite{AMVW15,BEGG} we computed also the backward error in terms of the coefficient of the polynomial. In particular, let $p(x)=\sum_{i=0}^n p_i x^i=\prod_{i=1}^n (x- \lambda_i)$ be our monic test polynomial with roots $\lambda_i$, and denote by $\tilde \lambda_i$  the computed roots obtained with our algorithm applied to the companion matrix $A$. We denote by $\hat p(x)$ the polynomial having $\tilde \lambda_i$ as exact roots, i.e., $\hat p(x)=\prod(x-\tilde \lambda_i)$. Using the extended precision arithmetic of Matlab we computed the coefficients $\hat p_i$ of $\hat p(x)$ in the monomial basis. We define the backward error in terms of the coefficients of the polynomial as follows
\begin{equation}\label{bep}
\mbox{bw}_{\mbox{err}}(p)=\max_{i}\frac{| p_i-\hat p_i|}{\|x\|_{\infty}},
\end{equation}
$x=(1, p+{n-1}, \ldots, p_0)$.

To confirm experimentally the stability of the algorithm we measured the backward error for matrices with prescribed norms. In particular, in Figure~\ref{fig:bs} for matrices in the class i.e.,  generic unitary-plus-rank-5, and unitary diagonal-plus-rank-5,  we report the results obtained on one thousand matrices of size $50$  with norm ranging from 1  to $10^{13}$, and we plot the absolute backward error $\| P_1^TA P_1-\tilde A^{(m)}\|_{\infty}$ versus $\|A\|_{\infty}$. The dashed lines represent a slope proportional to $\epsilon \|A\|_{\infty}$ and as we can see the plots agree with the results proved in Section~\ref{sec:back}. 

\begin{table}
	\begin{center}
		\begin{tabular}{lrrrrrrr}
			\hline
			$Name$         & n   &   k & degree & $\|\rm{ceig}(A)\|_{\infty}$ & $ \|A\|_{\infty} $ & $\mbox{forw}_{\mbox{err}}$ & $\mbox{back}_{\mbox{err}}$ \\ \hline
			acousticwave1d & 20  &  10 & 2      & 5.96e+01                       & 1.37e+01           & 1.42e-15                   & 1.02e-14                   \\
			bicycle        & 4   &   2 & 2      & 5.70e+02                       & 9.62e+03           & 2.60e-15                   & 8.29e-16                   \\
			cdplayer       & 120 &  60 & 2      & 4.50e+03                       & 2.67e+07           & 5.17e-16                   & 5.85e-15                   \\
			closedloop     & 4   &   2 & 2      & 9.00e+00                       & 3.00e+00           & 8.99e-16                   & 1.42e-15                   \\
			dirac          & 160 &  80 & 2      & 2.11e+03                       & 1.38e+03           & 5.24e-14                   & 1.39e-13                   \\
			hospital       & 48  &  24 & 2      & 4.49e+01                       & 1.11e+04           & 7.84e-13                   & 2.57e-14                   \\
			metalstrip     & 18  &   9 & 2      & 1.71e+02                       & 3.48e+02           & 7.78e-16                   & 2.42e-15                   \\
			omnicam1       & 18  &   9 & 2      & 5.04e+15                       & 1.73e+05           & 4.03e-07                   & 2.60e-15                   \\
			omnicam2       & 30  &  15 & 2      & 4.66e+17                       & 6.22e+07           & 1.16e-02                   & 4.90e-15                   \\
			powerplant     & 24  &   8 & 2      & 1.72e+05                       & 3.73e+07           & 7.13e-08                   & 2.68e-15                   \\
			qep2           & 6   &   3 & 2      & 1.80e+16                       & 4.00e+00           & 3.31e-09                   & 3.65e-16                   \\
			sign1          & 162 &  81 & 2      & 3.29e+09                       & 1.53e+01           & 4.10e-09                   & 5.84e-14                   \\
			sign2          & 162 &  81 & 2      & 9.61e+02                       & 5.63e+01           & 4.27e-13                   & 3.54e-14                   \\
			spring         & 10  &   5 & 2      & 2.33e+00                       & 8.23e+01           & 3.00e-16                   & 1.93e-15                   \\
			wiresaw1       & 20  &  10 & 2      & 1.57e+01                       & 1.42e+03           & 6.00e-14                   & 4.20e-15                   \\ \hline
			butterfly      & 240 &  64 & 4      & 2.97e+01                       & 5.18e+01           & 5.15e-14                   & 1.29e-13                   \\
			orrsommerfeld  & 40  &  10 & 4      & 1.88e+06                       & 9.67e+00           & 1.83e-14                   & 6.35e-15                   \\
			plasmadrift    & 384 & 128 & 3      & 6.64e+04                       & 3.24e+02           & 1.02e-13                   & 4.86e-14                   \\ \hline
		\end{tabular}
		\caption{Results on the NLEVP collection. On the  top part of the table results for quadratic problems with $k=n/2$.  On the bottom matrix polynomials with degree greater than 2.  \label{tab:NLEP}}
	\end{center}
\end{table}
In  Figure~\ref{fig:bsZ}, for matrix polynomials,  we compare backwards stability of the algorithm using implicit $Z$ or   updating explicitly $Z$ at each iteration.  We observe that the absolute backward error behaves as $\varepsilon \|A\|$ when $Z$ is updated at each iterations, while it b ehaves as $\varepsilon \|A\|^2$ when $Z$ is retrieved only at the end of the computations. This  shows in a very clear way  that it is better to  update  the rank-$k$ part at each iteration.  { In order to explain these discrepancies  theoretically we  recall that at the beginning of our error analysis  in the previous section we assume that  the matrix $\hat A$ is upper Hessenberg. However  the actual  matrix  obtained at the end of the Hessenberg reduction  process  only satisfies this requirement up to a backward error of order  $\varepsilon\|A\|$.  The different behavior of the explicit and  the implicit  algorithm  depends  on the propagation of this error. Specifically we can show that in the explicit variant  the error propagates additively  whereas in the implicit counterpart the error  increases  by a factor of order $\|A\|$.  Similar  error bounds have appeared in~\cite{AMRVW17}  where  a backward stable method for eigenvalues and eigenvectors approximation of matrix polynomials is proposed. The algorithm is a variant of Francis's implicitly shifted QR algorithm applied on the companion pencil, and the rank correction is not explicitly computed but it is computed only once at the end of the computation when retrieving the eigenvalues. The authors of~\cite{AMRVW17} proved that on the unscaled pencil $(S, T)$ the computed Schur form is the exact Schur form of a perturbed pencil $(S+\delta_S, T+\delta_T)$, where $\|\delta_S\|\le \varepsilon\|S\|^2$ and $\|\delta_T\|\le \varepsilon\|T\|^2$. Working with the pencil they are able to remove the dependence from the norm by scaling the pencil. In our case it is not possible to scale $A$ without destroying the unitary plus low rank structure, but we prove  that the absolute error is $O(\|A\|)\varepsilon$ keeping $Z$ explicit. In specific cases  as for polynomial rootfinding  where the Hessenberg structure of $\hat A$ is   determined exactly we achieve very good  results  also when keeping $Z$ implicit.}
\begin{table}
\begin{center}
	\begin{tabular}{rrrrrrrr}
		\hline
		n  &  k & degree &  $ \|A\|_{\infty} $&$\|\rm{ceig}(A)\|_{\infty}$  & $\mbox{forw}_{\mbox{err}}$ & $\mbox{back}_{\mbox{err}}$ \\ \hline
		50 &  2 &     25 & 2.46e+01                       & 1.77e+01           & 3.63e-15                   & 1.10e-14                   \\
		50 &  2 &     25 & 1.42e+06                       & 2.85e+01           & 2.65e-12                   & 9.37e-15                   \\
		50 &  5 &     10 & 2.34e+01                       & 2.32e+01           & 2.26e-15                   & 8.11e-15                   \\
		50 &  5 &     10 & 2.23e+06                       & 3.86e+01           & 7.73e-12                   & 8.31e-15                   \\
		50 & 10 &      5 & 2.75e+01                       & 3.24e+01           & 1.78e-15                   & 9.25e-15                   \\ 
		50& 10&        5& 3.16e+06                       &3.90e+01             & 1.08e-11                      & 7.78e-15  \\ \hline
		100& 5 &20 &1.02e+06  &3.43e+01 &   7.05e-13 & 9.88e-15 \\
		200& 5 &40 & 1.96e+06 & 3.73e+01&  2.62e-13 & 1.93e-14 \\
		400& 5 &80& 3.90e+06 & 1.01e+02&1.44e-13 & 3.19e-14 \\
		750& 5 &130& 7.09e+06 & 9.79e+01& 5.59e-13 & 5.30e-14 \\
		1000& 5 &200&  9.57e+07 & 2.60e+04& 8.49e-11 & 7.53e-14 \\
\hline
	\end{tabular}
	\caption{Top: Random polynomials of low degree with different norm sizes. We see that, in agreement with the theoretical results, the relative backward error is not affected by the norm of the matrix. Bottom: Random polynomials with higher degree and moderately high norm. We see that also in the larger example the stability is not  compromised. The figures for the larger tests  are the average over 10 runs.  \label{tab:rand_pol} }
\end{center}
\end{table}
In Table~\ref{tab:sc_pol} we report the backward errors in the scalar  polynomials described in Table~\ref{tab:scal_pol}. We report both the backward error in terms of the matrix coefficients and of the coefficients of the polynomial and we  see that the tests confirm the backward stability of the algorithm. Edelman and Murakami~\cite{EM95}  proved that  the analysis of the backward error in terms of the polynomial coefficients might introduce an additional factor proportional to $\|A\|$ but Table~\ref{tab:sc_pol} revels that we do  better than expected because $\mbox{bw}_{\mbox{err}}(p)=\varepsilon\, O(\|A\|)$, and not $\varepsilon\, O(\|A\|^2)$.
We report also the values of $\mbox{bw}_{\mbox{err}}(p)$  obtained on the same tests by two specialized algorithms for polynomial rootfinding, namely AMVW~\cite{AMVW15} and BEGG~\cite{BEGG} and by ZHSEQR, the LAPACK routine for computing  the eigenvalues of a Hessenberg matrix without any further structure. We obtain better results than those one gets using BEGG method, but 
sometimes we  lose a digit of precision compared to AMVW. We think that this is mostly due  to  differences in shift and in deflation criteria or in the retrieving,  in hight precision,   the coefficients of the polynomial $\hat p(x)$ from the computed roots.
Our method provides a unified framework to treat a larger class of matrices that contains companions and block companions but also  perturbations of CMV shapes, or unitary diagonal plus low rank, and so on. See~\cite{FS03} for some real world applications different from  scalar/matrix polynomials computation. 
\begin{figure}
	\begin{center}		
	\includegraphics[width=1\linewidth]{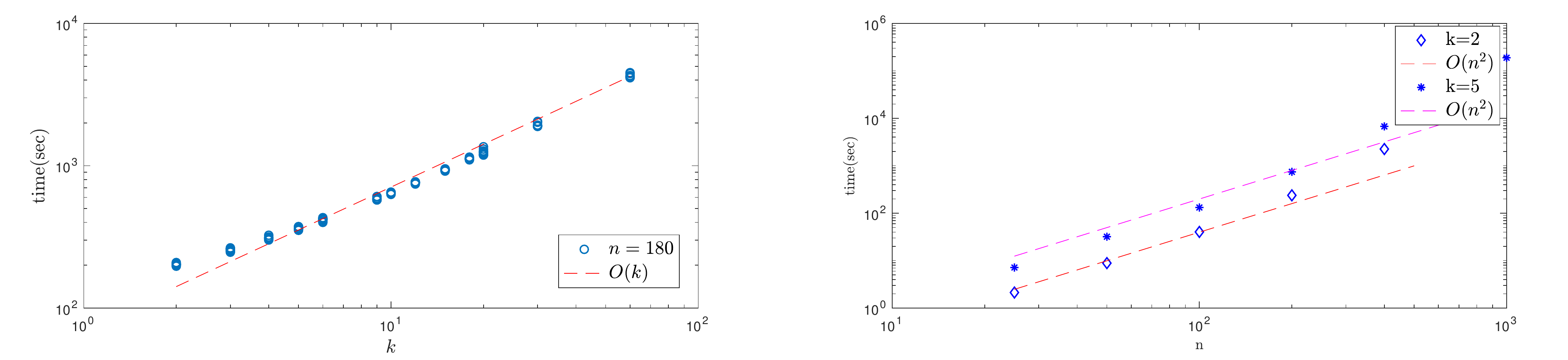}	
\end{center}
	\caption{On the right the double logarithmic plot for random matrices of size 180 that are unitary-plus-rank-$k$ with $k$ ranging from 1 to 60. The reference line shows the linear dependence on $k$. On the right, for $k=2$ and $k=5$ and matrices of size ranging from 25 to 1000. The dashed lines represent the $O(n^2)$ slope.}
\end{figure}
%
%
Table~\ref{tab:NLEP} reports the results obtained for several problems form the NLEVP collection~\cite{NLEVP}, which contains  polynomial eigenvalue problems from real-life applications.  To apply our method we needed to invert the coefficient corresponding to the higher degree of the polynomial so not all the problems of the collection were suitable for our algorithm.  In the collection we find mainly quadratic polynomial and a few examples of polynomial of degree $\ge 3$.   Table~\ref{tab:NLEP}  reportsthe degree $d$ of the polynomials, the size $k$ of the  coefficients, and $n=kd$ that is the size of the matrix of the linearization. We cannot compare directly with the method proposed in~\cite{AMRVW17} since the authors of that paper work on a pencil $(S, T)$ and then were able to scale each matrix of the pencil by a factor $\alpha=\max(\|S\|, \|T\| \}$  to remove the dependence of the error on the norm. In principle the algorithm BEGG~\cite{BEGG}, based on quasiseparable representation as well as other methods based on Givens weight and Givens vector representation, can be extended to deal with these matrices but with a cost of order at least $O(n^2k^3)$ which is not competitive for $k=O(n)$.
\begin{table}
	\begin{center}
		\begin{tabular}{rrrrrr}
			\hline
			n   & k  & $ \|A\|_{\infty} $ & $\|\rm{ceig}(A)\|_{\infty}$ & $\mbox{forw}_{\mbox{err}}$ & $\mbox{back}_{\mbox{err}}$ \\ \hline
			50  & 1  & 7.14e+00           & 4.19e+00                       & 7.33e-15                   & 9.45e-15                   \\
			50  & 1  & 9.70e+04           & 3.52e+01                       & 1.70e-16                   & 3.10e-15                   \\ \hline
			50  & 2  & 7.19e+00           & 3.78e+00                       & 7.51e-15                   & 9.92e-15                   \\
			50  & 2  & 9.83e+04           & 3.27e+01                       & 1.97e-16                   & 2.91e-15                   \\ \hline
			50  & 25 & 8.27e+00           & 1.00e+15                       & 5.36e-15                   & 1.10e-14                   \\
			50  & 25 & 9.98e+04           & 7.85e+14                       & 1.95e-16                   & 3.15e-15                   \\ \hline
			100 & 1  & 1.00e+01           & 1.37e+01                       & 1.46e-14                   & 1.79e-14                   \\
			100 & 1  & 9.85e+04           & 2.04e+02                       & 2.46e-16                   & 4.86e-15                   \\ \hline
			100 & 2  & 1.01e+01           & 1.97e+01                       & 1.52e-14                   & 1.89e-14                   \\
			100 & 2  & 9.91e+04           & 1.80e+02                       & 2.60e-16                   & 4.78e-15                   \\ \hline
			100 & 25 & 1.10e+01           & 2.18e+07                       & 1.30e-14                   & 2.07e-14                   \\
			100 & 25 & 9.99e+04           & 1.79e+06                       & 3.30e-16                   & 5.21e-15                   \\ \hline
		\end{tabular}
		\caption{Unitary plus low rank random matrices, with different sizes, rank of the correction and norm of the matrix. For each $n$ and $k$ we tested two cases $\|A\|_{\infty}=O(1)$ and $\|A\|_{\infty}=O(10^4)$. Each result reported is the average over 50 random tests. \label{tab:unit}  }
	\end{center}
\end{table}
The results of our algorithm fore higher degree random  matrix polynomials are reported in Table~\ref{tab:rand_pol} where  also the forward and backward errors for different values of the norm of the coefficients of the polynomials are shown. Each line refers to the average value over 50 tests on generalized companion matrices associated to matrix polynomials of size $k$ and degree $d=n/k$. For each pair $(k,d)$ we performed experiments varying the norm of the resulting generalized companion matrix. 
We see that as expected,  for matrices with larger norm, we may have a loss of accuracy in the computed solutions.


\begin{table}
	\centering
	\begin{tabular}{rrrrrr}
		\hline
		$n$ & $k$ & $ \|A\|_{\infty} $ & $\|\rm{ceig}(A)\|_{\infty}$ & $\mbox{forw}_{\mbox{err}}$ & $\mbox{back}_{\mbox{err}}$ \\ \hline
		50  & 1   & 3.46e+01           & 2.25e+00                       & 2.66e-16 & 2.78e-15   \\
		50  & 1   & 3.34e+06           & 2.25e+00                       & 3.11e-17 & 2.32e-15   \\ \hline
		50  & 2   & 5.96e+01           & 1.40e+01                       & 1.57e-16 & 2.78e-15   \\
		50  & 2   & 5.86e+06           & 2.18e+01                       & 8.83e-14 & 2.63e-15   \\ \hline
		50  & 25  & 6.37e+02           & 5.59e+01                       & 7.22e-17 & 2.58e-15   \\
		50  & 25  & 6.36e+07           & 9.34e+01                       & 8.63e-13 & 2.21e-15   \\ \hline
	\end{tabular}
	\caption{Unitary diagonal plus low rank random matrices, with different rank of the correction and norm of the matrix. For each $n$ and $k$ we tested two cases $\|A\|_{\infty}=O(1)$ and $\|A\|_{\infty}=O(10^4)$. Each result reported is the average over 50 random tests. \label{tab:diag}}
\end{table}

\begin{table}
	\begin{center}
		\begin{tabular}{rrrrrrr}
			\hline
			\#                  & n  & k  & $\|\rm{ceig}(A)\|_{\infty}$ & $ \|A\|_{\infty} $ & $\mbox{forw}_{\mbox{err}}$ & $\mbox{back}_{\mbox{err}}$ \\ \hline
			1                   & 20 & 10 & 1.87e+21                       & 2.30e+19           & 7.17e-01                   & 2.88e-15                   \\
			2                   & 20 & 10 & 2.08e+04                       & 9.48e+02           & 2.80e-01                   & 3.68e-15                   \\
			3                   & 20 & 10 & 2.45e+15                       & 1.52e+01           & 2.45e-01                   & 4.39e-15                   \\
			4                   & 20 & 10 & 3.19e+15                       & 2.67e+14           & 7.70e-02                   & 5.71e-15                   \\
			5                   & 20 & 10 & 4.72e+20                       & 1.66e+18           & 7.02e-02                   & 2.17e-15                   \\
			6                   & 30 & 15 & 1.11e+18                       & 1.05e+02           & 3.08e-01                   & 1.31e-14                   \\
			7                   & 20 & 10 & 4.90e+00                       & 5.18e+00           & 1.97e-15                   & 7.31e-15                   \\
			8                  & 20 & 10 & 2.21e+15                       & 7.10e+04           & 7.03e-11                   & 5.16e-15                   \\
			9                  & 40 & 3  & 8.06e+15                       & 4.04e+00           & 1.28e-03                   & 1.02e-14                   \\
			10                  & 30 & 15 & 3.83e+01                       & 1.06e+01           & 3.51e-15                   & 1.08e-14                   \\
			11                  & 29 & 15 & 2.37e+00                       & 4.00e+00           & 3.03e-15                   & 1.88e-14                   \\ \hline
		\end{tabular}
	\end{center}
	\caption{Results on Fiedler pentadiagonal matrices~\cite{CM2} associated to scalar polynomials. As proved in~\cite{DPRV19,DPRV_sub} the rank-correction  for dense polynomials is in general $k=\lceil n/2\rceil$ but it can be lower in the case the polynomial is sparse. \label{tab:fiedler}  }
\end{table}

Tables~\ref{tab:unit},~\ref{tab:diag}~\ref{tab:fiedler}, contain the results  for random unitary-plus-low-rank matrices, for perturbed unitary diagonal matrices and for Fiedler pentadiagonal matrices. In all the cases, and independently on the matrix norm, we get very good results for the backward stability. Note that  when the actual eigenvalues are unknown the results for the forward error show that the computed approximations  agree with those returned by Matlab {\tt eig} command.

\section{Conclusions}\label{sec:conclusions}

In this paper we have presented a novel algorithm for eigenvalue computation of unitary-plus-low-rank Hessenberg matrices.
The algorithm is computationally efficient with  respect  to both  the size of the matrix and   the size of the perturbation.
Further, the algorithm is shown  to be  backward stable.  At the core of the algorithm  is
a compressed data-sparse representation of the matrix as a product of
three factors. The outermost factors are unitary generalized  Hessenberg matrices whereas the factor in the middle is a unitary upper Hessenberg matrix corrected by a low rank perturbation located in the first rows. In particular cases it is possible to obtain the data-sparse Hessenberg form with cost $O(n^2k)$ flops instead of the customary $O(n^3)$ flops.  
It is shown that deflation and convergence of the QR iteration can be checked directly from the representation by
greatly simplifying the resulting fast scheme.  Future work is concerned with
the analysis  of efficient procedures   for computing the factored representation of the initial matrix as well as the design of
a fast QZ iteration for matrix pencils.


\end{document}